\documentclass[11pt,reqno]{amsart}
\usepackage[T1]{fontenc}
\usepackage{lmodern}
\usepackage{microtype}
\usepackage{mathtools,amssymb}
\usepackage{enumitem}
\usepackage{needspace}
\usepackage[a4paper,textwidth=158mm,textheight=237mm,centering]{geometry}
\usepackage[colorlinks=false,pdfborder={0 0 0}]{hyperref}
\hypersetup{
 pdftitle={Three Problems on Separable Quotients of Precompact Abelian Groups},
 pdfauthor={Li-Hong Xie and Jiang Yang},
 pdfsubject={Connected Baire reflexive precompact groups and finite-power quotients},
 pdfkeywords={precompact abelian group, separable quotient, connectedness, Baire, Pontryagin reflexivity}}
\allowdisplaybreaks[1]
\setlist[enumerate,1]{label=\textup{(\roman*)},leftmargin=2.2em,itemsep=3pt,topsep=5pt}
\newtheorem{theorem}{Theorem}[section]
\newtheorem{proposition}[theorem]{Proposition}
\newtheorem{lemma}[theorem]{Lemma}
\newtheorem{corollary}[theorem]{Corollary}
\theoremstyle{definition}

\newtheorem{problem}{Problem}
\theoremstyle{remark}
\newtheorem{remark}[theorem]{Remark}
\newtheorem{example}[theorem]{Example}
\numberwithin{equation}{section}
\newcommand{\T}{\mathbb T}
\newcommand{\R}{\mathbb R}
\newcommand{\Q}{\mathbb Q}
\newcommand{\Z}{\mathbb Z}
\newcommand{\N}{\mathbb N}
\newcommand{\cc}{\mathfrak c}
\newcommand{\cl}{\operatorname{cl}}
\newcommand{\Int}{\operatorname{int}}
\newcommand{\supp}{\operatorname{supp}}
\newcommand{\id}{\operatorname{id}}
\newcommand{\Hom}{\operatorname{Hom}}
\newcommand{\Sat}{\operatorname{Sat}}

\newcommand{\spanQ}{\operatorname{span}_{\Q}}
\newcommand{\res}{\mathbin{\upharpoonright}}
\newcommand{\gen}[1]{\langle #1\rangle}
\newcommand{\GL}{\operatorname{GL}}
\newcommand{\doi}[1]{\href{https://doi.org/#1}{doi:\nolinkurl{#1}}}

\title[Three problems on separable quotients]{Three Problems on Separable Quotients of Precompact Abelian Groups}
\author{Li-Hong Xie}
\address{School of Mathematics and Computational Science, Wuyi University,
Jiangmen, Guangdong 529000, P.R. China}
\email{yunli198282@126.com}
\author{Jiang Yang}
\address{School of Mathematical Sciences, Guangxi Minzu University,
Nanning 530006, P.R. China}
\email{yangjiangdy@126.com}
\thanks{Jiang Yang is the corresponding author.}
\date{}

\begin{document}
\begin{abstract}
We address three problems on separable quotients of topological
groups posed by Leiderman, Morris, and Tkachenko in \cite{LMT} published on Israel Journal of Mathematics. First, we construct in
ZFC a connected Baire Pontryagin-reflexive dense subgroup of $\T^{\cc}$
whose countable subgroups are $h$-embedded and whose uncountable subgroups
are dense. Its underlying abstract group is the circle group, and all its
compact subsets are finite. Second, we construct a zero-dimensional
Baire Pontryagin-reflexive example with the same subgroup properties whose
underlying group is free abelian of rank $\cc$. Both examples have no
nontrivial separable Hausdorff quotient. Third, for the group constructed in their Theorem~3.5, we
determine every closed subgroup of every finite power up to an integral
change of coordinates and prove that every countable subgroup of every
Hausdorff quotient of a finite power is $h$-embedded and closed. The same
conclusions hold for our free Baire reflexive example. These results answer Problem~1.25 negatively, Problem~3.12 affirmatively and realize all three regularity properties in
Problem~3.14 simultaneously in \cite{LMT}.
\end{abstract}
\maketitle

\section{Introduction}\label{sec:intro}

Throughout the paper, groups are Hausdorff and quotient groups carry the
quotient topology. Thus a quotient homomorphism is continuous, surjective,
and open. A continuous homomorphic image with a separately specified
topology need not be a quotient in this sense. This distinction is
particularly relevant for precompact abelian groups, whose continuous
characters take values in the metrizable circle group $\T$.

The separable quotient problem asks how much of a topological structure
can be retained in a separable quotient. Its classical Banach-space form
asks whether every infinite-dimensional Banach space has an
infinite-dimensional separable quotient( see \cite[Problem~1.1]{LMT} and the Scottish Book \cite{Mauldin2015}).
At the time of the account in \cite{LMT}, the general problem was open. The answer is positive for reflexive Banach spaces: every infinite-dimensional reflexive Banach space has an infinite-dimensional separable quotient, originally due to Pe{\l}czy\'{n}ski \cite{Pelczynski1964}; see also \cite[Corollary~1.9]{LMT} and the stronger weakly compactly generated case of Amir and Lindenstrauss \cite{AmirLindenstrauss1968}. The analogous question for locally convex spaces has a more varied history: large positive classes coexist with counterexamples. Positive results begin with Eidelheit \cite{Eidelheit1936} and Robertson \cite{Robertson1989}; counterexamples were constructed by K\c{a}kol, Saxon and Todd \cite{KakolSaxonTodd2014}. This contrast led Leiderman, Morris, and Tkachenko \cite{LMT} to formulate a systematic separable quotient program for topological groups.

Passing from vector spaces to groups changes the problem in two ways. There is no linear dimension with which to exclude a small quotient, so one must distinguish nontrivial from infinite. Moreover, separability and metrizability are independent requirements for general topological groups. This produces four natural quotient questions, stated systematically in \cite[Problems~1.21--1.24]{LMT}.

The first problem considered here was motivated by a striking positive analogy. Reflexive Banach spaces have separable infinite-dimensional quotients, originally due to Pe{\l}czy\'{n}ski \cite{Pelczynski1964}, and locally compact abelian groups are Pontryagin reflexive and also satisfy the corresponding group-theoretic quotient assertions, proved by Leiderman, Morris and Tkachenko \cite[Theorem~2.13]{LMT}. They therefore asked whether Pontryagin reflexivity alone is enough.

\Needspace{9\baselineskip}
\begin{problem}[Separable quotients of reflexive groups; {\cite[Problem~1.25]{LMT}}]
\label{prob:reflexive}
Does every infinite reflexive abelian topological group $G$ have a
separable quotient group which is
\begin{enumerate}
\item nontrivial;
\item infinite;
\item metrizable;
\item infinite and metrizable?
\end{enumerate}
\end{problem}

The other two problems originate in the precompact counterexample of
\cite[Theorem~3.5]{LMT}. Recall that a subgroup $D\leq A$ is
\emph{$h$-embedded in $A$} if every abstract homomorphism $D\to\T$
extends to a continuous character of $A$. The construction in that
theorem gives an uncountable dense subgroup $G\leq\T^{\cc}$ whose
countable subgroups are $h$-embedded and whose uncountable subgroups are
dense. These properties imply that every countable subgroup of $G$ is
closed and every nontrivial Hausdorff quotient of $G$ is nonseparable.
Leiderman, Morris, and Tkachenko further proved that every nontrivial
quotient of every power $G^\tau$, $\tau\geq1$, is nonseparable
\cite[Theorem~3.13]{LMT}.

Their finite-power argument controls countable subgroups of certain
character images; see the proof of~\cite[Theorem~3.11]{LMT}. Passing to
an arbitrary quotient of $G^k$ requires an additional argument, since
the inverse image of a countable subgroup can be uncountable. This leads
to the following question.

\begin{problem}[Quotients of finite powers; {\cite[Problem~3.12]{LMT}}]
\label{prob:powers}
Let $G$ be the group constructed in Theorem~3.5 of~\cite{LMT}, let
$k\geq1$ be an integer, and let $H$ be a quotient group of $G^k$.
Is every countable subgroup of $H$ $h$-embedded, or at least closed?
\end{problem}

The same construction raises a second issue. Its subgroup properties
exclude nontrivial separable quotients, but the example is
zero-dimensional. Can the subgroup properties responsible for this
obstruction coexist with connectedness, the Baire property, or Pontryagin
reflexivity? The published article formulates this question as follows.

\begin{problem}[Regularity of the precompact example; {\cite[Problem~3.14]{LMT}}]
\label{prob:regularity}
Does there exist a precompact abelian group $G$ as
in~\cite[Theorem~3.5]{LMT} which has one of the following additional
properties?
\begin{enumerate}[label=\textup{(\alph*)}]
\item $G$ is connected;
\item $G$ is Baire;
\item $G$ is reflexive.
\end{enumerate}
\end{problem}

In Problem~\ref{prob:regularity}, the phrase ``as in Theorem~3.5'' refers
to the subgroup properties and quotient obstruction described above;
zero-dimensionality is not imposed when asking for a connected example.
Part~(iii) of Problem~\ref{prob:reflexive}, read literally, allows the
trivial metrizable quotient.

We resolve these questions through two constructions and a structure
theorem. The first construction has underlying abstract group $\T$ and
is simultaneously connected, Baire, and Pontryagin reflexive. It retains
both subgroup properties of the original example, and hence has no
nontrivial separable Hausdorff quotient. It therefore realizes all three
regularity properties in Problem~\ref{prob:regularity} simultaneously and
gives negative answers to parts~(i), (ii), and~(iv) of
Problem~\ref{prob:reflexive}, as well as to the nontrivial-quotient
interpretation of part~(iii). The second construction is free abelian of
rank $\cc$, zero-dimensional, Baire, and Pontryagin reflexive, and has
the same subgroup properties. Thus the Baire and reflexive properties
can also be realized while retaining zero-dimensionality. Both
constructions are in ZFC.

For the particular free group constructed in~\cite[Theorem~3.5]{LMT},
we determine the closed subgroups of every finite power up to an integral
change of coordinates. The resulting normal form proves that every
countable subgroup of every Hausdorff quotient of a finite power is
$h$-embedded and closed, answering both alternatives in
Problem~\ref{prob:powers} affirmatively. The argument also applies to our
free Baire reflexive example; see Corollary~\ref{P:cor:free-powers}.

The duality ingredient used below is the published work of Ferrer,
Hern\'andez, Sep\'ulveda, and Trigos-Arrieta~\cite{FHST} relating the Baire
property to precompact duality. Related work of Peng~\cite{Peng} studies
the densities and weights of quotients of precompact abelian groups.
The present results concern the coexistence of regularity with the
subgroup properties above, and the structure of finite-power quotients.

This paper organized as follows: Section~\ref{sec:prelim} collects the common algebraic, category, and
duality preliminaries. The proof of Theorem~\ref{C:thm:main} occupies
Section~\ref{sec:connected}. Its algebraic recursion assigns separate
countable prescribed domains to the characters, while a fiberwise
criterion and a Cantor-set fusion argument establish connectedness.
Section~\ref{sec:free} begins with three auxiliary lemmas and then
states and proves Theorem~\ref{B:thm:main}. Its proof contains the entire
free-group construction, from independent selections in dense
$G_\delta$-sets to the verification of the subgroup, category, and
duality properties.
Section~\ref{sec:powers} proves Theorem~\ref{P:thm:structure} by
extracting a saturated sublattice of $\Z^k$ from the character rows
that have countable image on a closed subgroup. The answers to
Problems~\ref{prob:regularity}, \ref{prob:reflexive}, and
\ref{prob:powers} are given in Theorem~\ref{C:thm:main},
Theorem~\ref{B:thm:main}, and Theorem~\ref{P:thm:structure},
respectively.
\section{Common preliminaries}\label{sec:prelim}

All groups are abelian and written additively. We put
$\N=\{1,2,\ldots\}$ and $\omega=\{0,1,\ldots\}$. We write
$\Hom(A,\T)$ for abstract homomorphisms, without requiring continuity,
and $\gen{S}$ for the subgroup generated by $S$. For a set $X$, let
$A(X)$ be the abstract free abelian group with basis $X$; every element
has finite support. The letters used locally in
Sections~\ref{sec:connected}, \ref{sec:free}, and~\ref{sec:powers}
are reset in each section; in particular their groups denoted by $G$
are not identified. All topological groups are assumed Hausdorff.
The circle group $\T=\R/\Z$ carries its usual compact topology unless
another topology is explicitly specified.  We put
$\cc=2^{\aleph_0}$ and regard a cardinal as its initial ordinal.
A character of a topological abelian group is a continuous homomorphism
into $\T$.  The word countable includes finite sets.  Closures and interiors are
taken in the displayed ambient space; a superscript is added when
there could be ambiguity.  A Cantor set means a space homeomorphic
to $\{0,1\}^{\N}$.

We use the following standard terminology.  A group is
\emph{precompact} if its completion is compact, or, equivalently,
if it is topologically isomorphic to a subgroup of a compact group.
A space is \emph{Baire} if every countable intersection of dense open
sets is dense. A space is \emph{zero-dimensional} if it has a base
of sets that are both open and closed. A set is nowhere dense if its closure has empty
interior, and is meagre if it is a countable union of nowhere dense
sets.  We use the Baire theorem for compact Hausdorff spaces.
The countable chain condition, abbreviated ccc,
means that every pairwise disjoint family of nonempty open sets is
countable.  Background on topological groups and their completions
can be found in~\cite{AT}.  In particular, every Hausdorff topological
group is Tychonoff, that is, completely regular and Hausdorff.
The compact abelian duality facts used
below are treated in~\cite{Morris}.

Let  $\GL_k(\Z)=\{A\in M_k(\Z):\det A=\pm1\},$
where $M_k(\Z)$ is the set of all $k\times k$ integer matrices,
and the group operation is matrix multiplication.
For any abelian topological group $A$ and any matrix
$U=(u_{ij})\in\GL_k(\Z)$, write
\[
 \Phi_U:A^k\longrightarrow A^k,\qquad
 \Phi_U(x_1,\ldots,x_k)=
 \left(\sum_{j=1}^k u_{ij}x_j\right)_{i=1}^k.
\]
This is a topological automorphism with inverse $\Phi_{U^{-1}}$.
We use the convention $A^0=\{0\}$.

An abelian group $A$ is {\it divisible} if $nA=A$ for every positive
integer $n$. We will use the extension property of divisible groups
repeatedly, including when the homomorphisms are not continuous.

\begin{lemma}\cite[Theorem 1.1.2]{DPS}\label{C:lem:divisible}
Let $B$ be a subgroup of an abelian group $A$, and let $L$ be a
divisible abelian group.  Every homomorphism $u:B\to L$ extends to
a homomorphism $A\to L$.
\end{lemma}

Let $D=t(\T)=\Q/\Z$ be the torsion subgroup of $\T$.
Both $D$ and $\T$ are divisible.  Extend $\id_D$ to a retraction
$\T\to D$ and fix the resulting algebraic decomposition
\begin{equation}\label{C:eq:circle-decomposition}
 \T=D\oplus W.
\end{equation}
Here $W$ is divisible and torsion-free, hence a vector space over
$\Q$.  Its dimension is $\cc$: its cardinality is $\cc$, and an
infinite-dimensional vector space over the countable field $\Q$
has cardinality equal to the maximum of its dimension and
$\aleph_0$.  Throughout the construction, expressions such as $qx$
with $q\in\Q$ and $x\in W$ refer to this fixed vector space
structure. We do not choose a single-valued division operation on
all of $\T$.

If $E\leq\T$ is divisible and contains $D$, then
\begin{equation}\label{C:eq:divisible-domain}
 E=D\oplus(E\cap W),
\end{equation}
and $E\cap W$ is a $\Q$-subspace of $W$.
To see the assertion about rational multiples, choose in $E$ an
$n$th root of $x\in E\cap W$.  Its difference from the unique
$n$th root of $x$ in $W$ lies in $D\subseteq E$.
It follows in particular that every countable subgroup of $\T$
is contained in a countable divisible subgroup containing $D$:
take the $\Q$-span of its $W$-components and adjoin $D$.

For a subgroup $I\leq\T$, write
\[
 \Sat(I)=\{t\in\T:nt\in I\text{ for some integer }n\geq1\}.
\]
If $I$ is divisible, then
\begin{equation}\label{C:eq:saturation}
 \Sat(I)=I+D.
\end{equation}
Indeed, if $nt\in I$, choose $a\in I$ with $na=nt$; then
$t-a\in D$.  The reverse inclusion is immediate.  Thus, for a
divisible $I$, the condition $t\notin I+D$ says exactly that
$t+I$ has infinite order in $\T/I$.

\begin{lemma}\label{lem:closed}
If every countable subgroup of a Hausdorff abelian group $A$ is
$h$-embedded, then every countable subgroup of $A$ is closed.
\end{lemma}
\begin{proof}
Let $D\leq A$ be countable and let $x\in A\setminus D$. The cyclic
group $(D+\Z x)/D$ admits a homomorphism into $\T$ which is nonzero
at $x+D$. Its pullback to $D+\Z x$ extends to a continuous character
of $A$. The kernel of this character contains $D$ and excludes $x$.
Thus $D$ is closed.
\end{proof}

\begin{lemma}\label{lem:quotient-obstruction}
Let $A$ be an uncountable Hausdorff abelian group whose countable
subgroups are closed and whose uncountable subgroups are dense. Then
every nontrivial Hausdorff quotient of $A$ is nonseparable.
\end{lemma}
\begin{proof}
Let $q:A\to H$ be an open continuous surjection with $H\neq\{0\}$.
The proper closed subgroup $N=\ker q$ is countable. If $H$ were
separable, it would contain a countable dense subgroup $D$. Then
$q^{-1}(D)$ would be countable, hence closed in $A$. It is also dense,
since every nonempty open subset of $A$ has nonempty open image meeting
$D$. This would imply $q^{-1}(D)=A$, contrary to the uncountability
of $A$.
\end{proof}
\begin{lemma}\label{B:lem:countable}
Let $A$ be an abelian topological group in which every countable subgroup
is $h$-embedded. Then every compact subset of $A$ is finite.
\end{lemma}

\begin{proof}
Every countable subgroup of $A$ is closed by Lemma~\ref{lem:closed}.

Suppose that $C\subseteq A$ is infinite and compact. Choose a countably
infinite subset $S\subseteq C$ and put $D=\gen{S}$. The space
$C\cap D$ is infinite, countable and compact Hausdorff. Every countable
compact Hausdorff space is metrizable: choose one continuous real-valued
function separating each pair of distinct points and use the resulting
embedding into $[0,1]^\N$. Hence $C\cap D$ contains distinct points
$d_n$ converging to a point $d\in D$. After omitting at most one term,
we have $e_n=d_n-d\neq0$ for every $n$.

Consider the compact group
\[
 L=\Hom(D,\T)\subseteq\T^D
\]
with its pointwise topology and normalized Haar measure $\mu$.
Every $h\in L$ is continuous on $D$, since $D$ is $h$-embedded in
$A$. It follows that $h(e_n)\to0$ for every $h\in L$.
On the other hand, the function
$h\mapsto\exp(2\pi i h(e_n))$ is a nontrivial character of $L$.
Haar orthogonality and dominated convergence give
\[
 0=\lim_{n\to\infty}\int_L\exp(2\pi i h(e_n))\,d\mu(h)
   =\int_L1\,d\mu(h)=1,
\]
a contradiction.
\end{proof}

For a set $I$, let $\Z^{(I)}$ denote the group of finitely supported
integer families indexed by $I$.  A family $s=(s_i)\in\Z^{(I)}$
defines a character of $\T^I$ by
\begin{equation}\label{C:eq:integer-character}
 \chi_s(z)=\sum_{i\in I}s_i z_i.
\end{equation}
Every continuous character of $\T^I$ has this form.  One way to
see the finite-support assertion is to choose an arc about zero in
$\T$ containing no nontrivial subgroup.  Continuity at zero implies
that the image of a coordinate tail subgroup is contained in this
arc, and therefore is zero.  On the remaining finite product one
uses the familiar fact that the continuous characters of $\T$
are its integer multiples.

For a subgroup $M\leq\T^I$, compact abelian duality gives
\begin{equation}\label{C:eq:annihilator}
 \cl_{\T^I}M
 =\bigcap\{\ker\chi_s:s\in\Z^{(I)},\ \chi_s(M)=\{0\}\}.
\end{equation}
This follows by applying character separation to the compact
quotient by $\cl M$.  In particular, $M$ is dense precisely when
no nonzero integer character vanishes on it.  A continuous
character of a dense subgroup extends uniquely to the compact
completion: a continuous homomorphism is uniformly continuous,
and $\T$ is complete.  These observations explain why estimates
for all finite integer combinations of our coordinates will
control every continuous character of the group constructed below.

A nonzero $s\in\Z^{(I)}$ is \emph{primitive} if the greatest common
divisor of its coefficients is one.  An integer change of its
finitely many active coordinates sends a primitive character to a
coordinate projection.  Consequently, $\ker\chi_s$ is connected
for primitive $s$, and its annihilator is $\Z s$.  In general, a
fiber of a nonzero integer character is a finite disjoint union of
fibers of a primitive character.  Every nonzero integer character
is an open surjection onto $\T$; its fibers are closed and nowhere
dense.

\begin{lemma}\label{C:lem:finite-cover}
Let $I=I'\cup\{m\}$, where $m\notin I'$, let $s\in\Z^{(I)}$ satisfy $s_m\neq0$.
For $t\in\T$, put $L=\chi_s^{-1}(t)$, where $\chi_s$ is
defined by~\eqref{C:eq:integer-character}. The restriction to $L$ of the coordinate
projection $\pi:\T^I\to\T^{I'}$ is a finite covering map.
If $F$ is closed and nowhere dense in $L$, then $\pi(F)$ is closed
and nowhere dense in $\T^{I'}$.
\end{lemma}

\begin{proof}
Writing $a(z)=\sum_{i\in I'}s_i z_i$, the equation defining $L$ is
$s_m z_m=t-a(z)$.  Near any point of $\T^{I'}$, the right-hand side
lies in a sufficiently small arc of $\T$, on which the map
$y\mapsto s_m y$ admits $|s_m|$ disjoint continuous local inverses.
Their graphs give an evenly covered neighborhood.

The image $\pi(F)$ is closed because $F$ is compact.  In an evenly
covered neighborhood its image is the union of finitely many
closed nowhere dense sets, one from each covering sheet.  Such a
union has empty interior.  Thus $\pi(F)$ has empty interior
everywhere.
\end{proof}

We use only the following cardinal arithmetic facts:
\begin{equation}\label{C:eq:cardinals}
 \cc^{\aleph_0}=\cc,
 \qquad \operatorname{cf}(\cc)>\aleph_0.
\end{equation}
Here $\operatorname{cf}(\kappa)$ denotes the least cardinality of an
unbounded subset of the ordinal $\kappa$.  The second assertion is a
consequence of K\"onig's theorem; see~\cite{Jech}.
Thus every countable subset of the ordinal $\cc$ is bounded in
$\cc$.  A countable product of circles has a countable base and
at most $\cc$ closed subsets.  The set of pairs consisting of a
countable subgroup $C\leq\T$ and a homomorphism $C\to\T$ also has
cardinality $\cc$.  There are at most $\cc$ such pairs by the first
identity in~\eqref{C:eq:cardinals}, and already an infinite cyclic
subgroup has $\cc$ homomorphisms into $\T$.
No regularity of $\cc$ is assumed.

A compact group is ccc.  Indeed, every nonempty open set has
positive Haar measure, since finitely many of its translates cover
the group.  A pairwise disjoint family of positive-measure open
sets must be countable.

We say that $A\subseteq\T^I$ \emph{depends on} $J\subseteq I$ if
$A=p_J^{-1}(A')$ for some $A'\subseteq\T^J$.  A basic open cylinder
depends on finitely many coordinates.  Coordinate projections
are open, and hence taking the closure or the interior of a
cylinder preserves its coordinate support.  An open set $O$ is regular open
if $O=\Int\cl O$; a set is regular closed if it is the closure of
an open set.

\begin{lemma}\label{C:lem:countable-support}
Let $K=\T^I$.
\begin{enumerate}
\item Every regular open subset of $K$ depends on countably many
coordinates.
\item Every dense open subset of $K$ contains a dense open subset
which is a countable union of basic open cylinders.
\item If $G\leq K$ is dense and every countable coordinate projection
of $G$ is connected, then $G$ is connected.
\end{enumerate}
\end{lemma}

\begin{proof}
Inside an open set $O$, choose a maximal pairwise disjoint family
of nonempty basic open cylinders.  This family is countable by
ccc.  Its union $V$ is dense in $O$, since otherwise an additional
basic cylinder could be inserted.  If $O$ is regular open, then
$O=\Int_K\cl_K V$, which proves (1).  If $O$ is dense in $K$,
then $V$ is dense in $K$, proving (2).

For (3), suppose $G=A\cup B$ such that $A\cap B=\emptyset$ and $A, B$ are nonempty
clopen sets in $G$.  Put
\[
 U=K\setminus\cl_K B,
 \qquad V=K\setminus\cl_K A.
\]
Density of $G$ gives $U\cap V=\varnothing$, and relative openness
gives $U\cap G=A$ and $V\cap G=B$.  In particular,
$G\subseteq U\cup V$.  If $A=G\cap O$ with $O$ open in $K$, then
$\cl_K A=\cl_K O$; thus $\cl_K A$ is regular closed.  The same holds
for $B$, so $U,V$ are regular open.  By (1), both depend on one
common countable set $J$ of coordinates.  Their projections give
a nontrivial separation of $p_J(G)$, a contradiction.
\end{proof}

\begin{lemma}\label{B:lem:baire}
Let $\Gamma$ be nonempty and let $A$ be a dense subgroup of
$K=\T^\Gamma$. Suppose that
\begin{equation}\label{B:eq:category}
 p_J(A)\cap R\neq\varnothing,
\end{equation}
 whenever $J\subseteq\Gamma$ is nonempty and countable and
$R\subseteq\T^J$ is a dense $G_\delta$-set, where $p_J:\T^\Gamma\to\T^J$ is the
coordinate projection. Then $A$ is Baire.
\end{lemma}

\begin{proof}
First we show that $A$ is not meagre in $K$. Let $F_n$, $n\in\N$,
be closed nowhere dense subsets of $K$. By
Lemma~\ref{C:lem:countable-support}(ii), each $K\setminus F_n$ contains
a dense open set $W_n$ that is a countable union of basic cylinders.
Together, these cylinders involve only countably many coordinates.
Choose a nonempty countable $J$ containing
these coordinates. We may write
\[
 W_n=p_J^{-1}(V_n),
\]
where each $V_n$ is dense and open in $\T^J$. By
\eqref{B:eq:category}, some point of $A$ belongs to
$p_J^{-1}(\bigcap_n V_n)$ and hence avoids $\bigcup_n F_n$.
Since every meagre set is contained in a countable union of closed
nowhere dense sets, $A$ is not meagre in $K$.

If $A$ were not Baire, it would have a nonempty open subset $U$ which
is meagre in $A$. Since $A$ is dense in $K$, the closure in $K$ of
each nowhere dense subset of $A$ is nowhere dense in $K$. Consequently
$U$ would be meagre in $K$. Precompactness implies that finitely many
translates of $U$ cover $A$, making $A$ meagre in $K$, a contradiction.
\end{proof}

The last paragraph is the usual equivalence between the Baire property
of a precompact group and non-meagreness in its completion; see
Bruguera and Tkachenko~\cite{BT}. Notice that
\eqref{B:eq:category} concerns dense $G_\delta$-sets. Meeting all
nonempty $G_\delta$-sets would be a substantially stronger requirement.

Let $\widehat A$ denote the group of continuous characters $A\to\T$.
We write $\widehat A_p$ for this group with the topology of pointwise
convergence on $A$, and $\widehat A_c$ for the compact-open topology,
that is, the topology of uniform convergence on compact subsets of $A$.
The group $A$ is \emph{Pontryagin reflexive} if the evaluation map
\[
 \alpha_A:A\longrightarrow\widehat{\widehat A_c}_c,
 \qquad \alpha_A(a)(\chi)=\chi(a),
\]
is a topological isomorphism.

For a precompact abelian group, pointwise duality gives a canonical
topological isomorphism
\begin{equation}\label{B:eq:pointwise-duality}
 A\cong\widehat{\widehat A_p}_p;
\end{equation}
see Comfort and Ross~\cite{CR}, and the explicit account in
\cite{FHST}. This identity alone is not a statement
of Pontryagin reflexivity, because compact-open and pointwise topologies
need not agree. We use the following theorem in precisely that distinction.

\begin{theorem}[Ferrer--Hern\'andez--Sep\'ulveda--Trigos-Arrieta
{\cite{FHST}}]\label{B:thm:FHST}
Let $A$ be a precompact abelian group. If $\widehat A_p$ is Baire,
then every compact subset of $A$ is finite.
\end{theorem}

The following consequence of Theorem~\ref{B:thm:FHST} gives the
reflexivity criterion needed below; see also~\cite{FHST}. We include
the argument to specify both topologies involved.

\begin{corollary}\label{B:cor:reflexive}
A precompact Baire abelian group without infinite compact subsets is
Pontryagin reflexive.
\end{corollary}

\begin{proof}
Let $A$ be such a group and put $B=\widehat A_p$. Since all compact
subsets of $A$ are finite, $\widehat A_c=\widehat A_p=B$.
The group $B$ is precompact, and
$\widehat B_p\cong A$ by~\eqref{B:eq:pointwise-duality}. In particular,
$\widehat B_p$ is Baire. Theorem~\ref{B:thm:FHST}, applied to $B$,
shows that every compact subset of $B$ is finite. Thus
$\widehat B_c=\widehat B_p$, and~\eqref{B:eq:pointwise-duality}
identifies $\alpha_A$ with a topological isomorphism
$A\cong\widehat{\widehat A_c}_c$.
\end{proof}

\begin{lemma}\label{C:lem:cantor-selection}
Let $X$ be a nonempty compact metrizable space, and let
$\varphi_n:X\to Y_n$ $(n\geq1)$ be continuous maps into metrizable
spaces.  If $X$ cannot be covered by countably many fibers of these
maps, then it contains a Cantor set $P$ on which every $\varphi_n$
is one-to-one.
\end{lemma}

\begin{proof}
Let $\mathcal I$ consist of all subsets of $X$ which are contained
in a countable union of fibers of the maps $\varphi_n$.  This is
a $\sigma$-ideal.  Let $O$ be the union of all open members of
$\mathcal I$.  Since $X$ has a countable base, $O$ is itself in
$\mathcal I$.  Thus $F=X\setminus O$ is a nonempty compact set.

Every set $F\cap\varphi_n^{-1}(y)$ is closed and nowhere dense
in $F$.  Otherwise there is an open $V\subseteq X$ such that
\[
 \varnothing\neq F\cap V\subseteq\varphi_n^{-1}(y).
\]
Then $V\subseteq O\cup\varphi_n^{-1}(y)$, so $V\in\mathcal I$,
contrary to $F\cap V\neq\varnothing$.  In particular, $F$ has no
isolated points.

We construct a binary family of nonempty relatively open subsets of $F$ by induction.
For each $k \ge 1$, suppose for all $s \in \{0,1\}^{k-1}$ we have constructed $F_s$ such that:
\begin{enumerate}
    \item[(i)] if $s \neq t$, then $\overline{F_s} \cap \overline{F_t} = \emptyset$;
    \item[(ii)] $\text{diam}(F_s) \le 1/(k-1)$ for $k > 1$;
    \item[(iii)] $\varphi_i(\overline{F_s})$ are pairwise disjoint for all $i \le k-1$.
\end{enumerate}
For each $s \in \{0,1\}^{k-1}$, we choose two distinct points $x_{s0}, x_{s1} \in F_s$. Since $\varphi_i^{-1}(y) \cap F$ is closed and nowhere dense in $F$, and $F$ has no isolated points, $F_s$ cannot be covered by finitely many such fibers. Thus, at each choice only finitely many closed nowhere dense fibers must be avoided, allowing us to select all new points so that their images under \(\varphi_1, \dots, \varphi_k\) are pairwise distinct.

Since $Y_i$ are Hausdorff spaces and $\varphi_i(x_{s0}) \neq \varphi_i(x_{s1})$, there exist disjoint open neighborhoods $U_{i,0}, U_{i,1} \subset Y_i$ containing $\varphi_i(x_{s0})$ and $\varphi_i(x_{s1})$, respectively. By continuity of $\varphi_i$, the preimages $\varphi_i^{-1}(U_{i,0})$ and $\varphi_i^{-1}(U_{i,1})$ are disjoint open neighborhoods of $x_{s0}$ and $x_{s1}$. We choose $F_{s0} \subset F_s \cap \varphi_i^{-1}(U_{i,0})$ and $F_{s1} \subset F_s \cap \varphi_i^{-1}(U_{i,1})$ to be nonempty relatively open subsets of $F$ such that:
\begin{itemize}
    \item [(1)] $\overline{F_{s0}} \cap \overline{F_{s1}} = \emptyset$;
    \item [(2)] $\text{diam}(F_{s0}), \text{diam}(F_{s1}) \le 1/k$;
    \item [(3)] $\varphi_i(\overline{F_{s0}}) \cap \varphi_i(\overline{F_{s1}}) = \emptyset$ for all $i \le k$.
\end{itemize}

For any infinite binary sequence $\sigma = (\sigma_1, \sigma_2, \dots) \in \{0,1\}^{\mathbb{N}}$, consider the nested sequence of closed sets $\{\overline{F_{\sigma|_k}}\}_{k=1}^\infty$. Since $F$ is compact and the diameters tend to zero, Cantor's Intersection Theorem implies that $\bigcap_{k=1}^\infty \overline{F_{\sigma|_k}}$ is a singleton. Let $x_\sigma$ denote this unique point, and define
$$
P = \{x_\sigma : \sigma \in \{0,1\}^{\mathbb{N}}\} \subset F.
$$
Since $P$ is compact, has no isolated points, and totally disconnected, it is a Cantor set.

Finally, we verify that each $\varphi_n$ is one-to-one on $P$. Let $x_\sigma, x_\tau \in P$ with $x_\sigma \neq x_\tau$. For any fixed $n \ge 1$, since $\sigma \neq \tau$, there exists some level $k \ge n$ such that $\sigma|_k \neq \tau|_k$. By construction, $\overline{F_{\sigma|_k}}$ and $\overline{F_{\tau|_k}}$ are disjoint, and $\varphi_n(\overline{F_{\sigma|_k}}) \cap \varphi_n(\overline{F_{\tau|_k}}) = \emptyset$. Therefore, $\varphi_n(x_\sigma) \neq \varphi_n(x_\tau)$. Since $n$ was arbitrary, every $\varphi_n$ is one-to-one on $P$.
\end{proof}

For a nonempty countable set $J$, call a Cantor set
$P\subseteq\T^J$ \emph{character-injective} if every nonzero
integer character of $\T^J$ is one-to-one on $P$.
This expression will only be used with this explicit meaning.

\begin{corollary}\label{C:cor:fiber-cover}
Suppose $H\subseteq\T^J$ meets every character-injective Cantor
set, where $J$ is nonempty and countable.  Every compact set
disjoint from $H$ is covered by countably many fibers of nonzero
primitive integer characters.
\end{corollary}

\begin{proof}
There are only countably many integer characters on $\T^J$.
Apply Lemma~\ref{C:lem:cantor-selection} to their restrictions to the
compact set.  If it were not covered by countably many fibers,
it would contain a character-injective Cantor set disjoint from
$H$.  Finally split each nonprimitive character fiber into finitely
many primitive character fibers.
\end{proof}

\begin{lemma}\label{C:lem:cantor-gdelta}
Let $J$ be nonempty and countable.  If $O\subseteq\T^J$ is nonempty
and open and $V_n$ $(n\geq1)$ are dense open subsets of $\T^J$,
then $O\cap\bigcap_n V_n$ contains a character-injective Cantor set.
\end{lemma}

\begin{proof}
Enumerate the nonzero integer characters.  Carry out a binary
construction inside $O$, requiring the closures chosen at level
$k$ to lie in $V_1\cap\cdots\cap V_k$ and to have diameter at most
$2^{-k}$.  Choose the finitely many new points at that level so
that their values under each of the first $k$ characters are
pairwise distinct.  This is possible because character fibers
are closed and nowhere dense.  Shrink the neighborhoods to preserve
these distinctions on their closures.  The resulting Cantor set
has all the required properties.
\end{proof}

\section{A connected Baire reflexive example}\label{sec:connected}

We use the notation and common results of Section~\ref{sec:prelim}.

Let $\theta$ be a nonzero countable ordinal.  For
$1\leq\alpha\leq\theta$ write
\[
 K_\alpha=\T^{[0,\alpha)},
 \qquad
 \pi_\alpha^\beta:K_\beta\longrightarrow K_\alpha
 \quad(\alpha\leq\beta\leq\theta)
\]
for the coordinate projections.  For $H\subseteq K_\theta$, put
$H_\alpha=\pi_\alpha^\theta(H)$.

We first control the projections of compact sets missed by $H$;
then a fusion lemma handles countable limit stages.

\begin{lemma}\label{C:lem:obstacle}
Let $0<m<\theta$. Suppose that $H_{m+1}$ meets every
character-injective Cantor subset of $K_{m+1}$ and that, for every
primitive $s\in\Z^{([0,m+1))}$ with $s_m\neq0$ and every $t\in\T$,
the set $H_{m+1}\cap\chi_s^{-1}(t)$ is dense in $\chi_s^{-1}(t)$.
If $F\subseteq K_{m+1}$ is compact and
$F\cap H_{m+1}=\varnothing$, then $\pi_m^{m+1}(F)$ is closed and
nowhere dense in $K_m$.
\end{lemma}

\begin{proof}
By Corollary~\ref{C:cor:fiber-cover}, there are primitive characters
$\chi_{s_n}$ and points $t_n\in\T$ such that
\[
 F\subseteq\bigcup_{n\geq1}L_n,
 \qquad L_n=\chi_{s_n}^{-1}(t_n).
\]
Put $F_n=F\cap L_n$.  If $(s_n)_m=0$, the projection of $F_n$
is contained in a nonzero character fiber in $K_m$ and is therefore
nowhere dense.  If $(s_n)_m\neq0$, the assumed density on character fibers shows that
$F_n$ is closed and nowhere dense in $L_n$.  By
Lemma~\ref{C:lem:finite-cover}, its projection is again closed and
nowhere dense.  Thus $\pi_m^{m+1}(F)=\bigcup_n\pi_m^{m+1}(F_n)$
is meagre.  It is also compact, hence closed.  Since $K_m$ is
Baire, this closed meagre set has empty interior.
\end{proof}

Let $s = (s_1, s_2, \dots, s_n) \in \{0,1\}^n$  and let $i \in \{0,1\}$. We denote by $s^\frown i$
the sequence of length $n+1$ obtained by appending $i$ to the end of $s$, namely
$$
s^\frown i = (s_1, s_2, \dots, s_n, i) \in \{0,1\}^{n+1}.
$$
Here, the symbol $\frown$ denotes the \emph{concatenation} operation on finite sequences.
The next lemma gives the limit-stage argument.  Its statement
involves only open subsets of products of circles.

\begin{lemma}\label{C:lem:limit-fusion}
Let $U,V\subseteq K_\theta$ be disjoint nonempty open sets.  Set
\begin{equation}\label{C:eq:overlaps}
 U_\alpha=\pi_\alpha^\theta(U),\qquad
 V_\alpha=\pi_\alpha^\theta(V),\qquad
 W_\alpha=U_\alpha\cap V_\alpha
 \quad(1\leq\alpha\leq\theta).
\end{equation}
Suppose $W_1\neq\varnothing$ and
\begin{equation}\label{C:eq:successor-dense-overlap}
 \pi_\alpha^{\alpha+1}(W_{\alpha+1})
 \text{ is dense in }W_\alpha
 \quad(1\leq\alpha<\theta).
\end{equation}
Then some countable limit ordinal $\beta\leq\theta$ admits a
character-injective Cantor set
\begin{equation}\label{C:eq:cantor-obstacle}
 P\subseteq K_\beta\setminus(U_\beta\cup V_\beta).
\end{equation}
\end{lemma}

\begin{proof}
All the sets in~\eqref{C:eq:overlaps} are open, and
\begin{equation}\label{C:eq:overlap-monotone}
 \pi_\alpha^\beta(W_\beta)\subseteq W_\alpha
 \quad(\alpha\leq\beta).
\end{equation}
The final overlap $W_\theta$ is empty.  Since $W_1$ is nonempty,
there is a least ordinal $\beta\leq\theta$ such that, for some
$1\leq\gamma<\beta$,
\begin{equation}\label{C:eq:first-failure}
 \pi_\gamma^\beta(W_\beta)
 \text{ is not dense in }W_\gamma.
\end{equation}
Choose a nonempty open set $O\subseteq W_\gamma$ with
\begin{equation}\label{C:eq:forbidden-overlap}
 O\cap\cl_{K_\gamma}\bigl(\pi_\gamma^\beta(W_\beta)\bigr)
 =\varnothing.
\end{equation}

The ordinal $\beta$ cannot be a successor.  Indeed, if
$\beta=\delta+1$, then~\eqref{C:eq:successor-dense-overlap} gives
density of $\pi_\delta^\beta(W_\beta)$ in $W_\delta$.
When $\gamma<\delta$, minimality of $\beta$ also gives density of
$\pi_\gamma^\delta(W_\delta)$ in $W_\gamma$.  A continuous image
of a dense subset is dense in the image of the whole space, so
composition contradicts~\eqref{C:eq:first-failure}.  The case
$\gamma=\delta$ contradicts~\eqref{C:eq:successor-dense-overlap}
directly.  Thus $\beta$ is a countable limit ordinal.

For $\gamma\leq\alpha<\beta$, put
\[
 W_\alpha^O=W_\alpha\cap(\pi_\gamma^\alpha)^{-1}(O).
\]
These sets are nonempty and open.  Moreover, minimality of $\beta$
implies that
\begin{equation}\label{C:eq:localized-density}
 \pi_\alpha^{\alpha'}(W_{\alpha'}^O)
 \text{ is dense in }W_\alpha^O
 \quad(\gamma\leq\alpha<\alpha'<\beta).
\end{equation}
Indeed, first use the density of
$\pi_\alpha^{\alpha'}(W_{\alpha'})$ in $W_\alpha$, and then
restrict to the open cylinder above $O$.

Enumerate the nonzero integer characters of $K_\beta$ as
$\ell_1,\ell_2,\ldots$, and enumerate its coordinates as
$\xi_1,\xi_2,\ldots$.  Choose a strictly increasing cofinal
sequence
\[
 \gamma<\alpha_1<\alpha_2<\cdots<\beta
\]
such that $[0,\alpha_n)$ contains the supports of
$\ell_1,\ldots,\ell_n$ and the coordinates
$\xi_1,\ldots,\xi_n$.  Equip $K_\beta$ with the metric
\[
 d(z,z')=\sum_{r\geq1}2^{-r}\rho(z_{\xi_r},z'_{\xi_r}),
\]
where $\rho$ is a compatible metric on $\T$ bounded by one.

We construct nonempty open sets
$V_s\subseteq K_{\alpha_n}$ for $s\in\{0,1\}^n$ so that:
\begin{enumerate}
\item $\cl V_s\subseteq W_{\alpha_n}^O$, and, for $i\in\{0,1\}$,
\begin{equation}\label{C:eq:fusion-nesting}
 \cl V_{s^\frown i}
 \subseteq W_{\alpha_{n+1}}^O
 \cap(\pi_{\alpha_n}^{\alpha_{n+1}})^{-1}(V_s);
\end{equation}
\item for each $k\leq n$, the sets $\ell_k(\cl V_s)$,
$s\in\{0,1\}^n$, are pairwise disjoint;
\item the sets
$(\pi_{\alpha_n}^{\beta})^{-1}(\cl V_s)$ have $d$-diameter
at most $2^{1-n}$.
\end{enumerate}
Here a character whose support lies below $\alpha_n$ is also
regarded as a character on $K_{\alpha_n}$.

To carry out a step,~\eqref{C:eq:localized-density} ensures that each
parent set has a nonempty open inverse image in
$W_{\alpha_{n+1}}^O$.  Choose two points above each parent.  Choose
all the points at this level so that their values under each of
the first $n+1$ characters are pairwise distinct.  This requires
avoiding only finitely many closed nowhere dense character fibers
at each choice.  Next choose sufficiently small neighborhoods of
the selected points.  Regularity and continuity give (i) and (ii).
For (iii), make the coordinate diameters on
$\xi_1,\ldots,\xi_{n+1}$ at most $2^{-(n+1)}$; the total weight
of the remaining coordinates is at most $2^{-(n+1)}$.
The initial level is obtained by the same argument inside
$W_{\alpha_1}^O$.

Define
\begin{equation}\label{C:eq:fusion-cantor}
 P=\bigcap_{n\geq1}\
   \bigcup_{s\in\{0,1\}^n}
    (\pi_{\alpha_n}^{\beta})^{-1}(\cl V_s).
\end{equation}
The sets along any branch form a decreasing sequence of nonempty
compact sets with diameters tending to zero.  Their intersection
is one point.  Distinct branches give distinct points by (ii),
and the resulting map from $\{0,1\}^{\N}$ to $P$ is continuous.
Consequently $P$ is a Cantor set.  For a fixed $k$, two distinct
branches are separated at some level $n\geq k$, where (ii) shows
that their $\ell_k$-values are distinct.  Thus $P$ is
character-injective.

It remains to locate $P$.  If $z\in P$, then
$\pi_{\alpha_n}^\beta(z)\in W_{\alpha_n}$ for every $n$.
Every basic neighborhood of $z$ is determined by finitely many
coordinates, all lying below some $\alpha_n$.  Since that prefix
of $z$ belongs to both $U_{\alpha_n}$ and $V_{\alpha_n}$, the
neighborhood meets both $U_\beta$ and $V_\beta$.  Hence
\begin{equation}\label{C:eq:both-closures}
 z\in\cl U_\beta\cap\cl V_\beta.
\end{equation}
Also $\pi_\gamma^\beta(z)\in O$.  By
\eqref{C:eq:forbidden-overlap}, the open cylinder
$C=(\pi_\gamma^\beta)^{-1}(O)$ contains no point of
$U_\beta\cap V_\beta$.  If $z$ belonged to $U_\beta$, then the
open neighborhood $U_\beta\cap C$ of $z$ would meet $V_\beta$ by
\eqref{C:eq:both-closures}, a contradiction.  The same argument
excludes $z\in V_\beta$.  This proves~\eqref{C:eq:cantor-obstacle}.
\end{proof}

\begin{lemma}\label{C:thm:criterion}
A dense subgroup $H\leq K_\theta$ is connected if:
\begin{enumerate}
\item $H_1=\T$;
\item for every $1\leq\alpha\leq\theta$, the set $H_\alpha$ meets
every character-injective Cantor subset of $K_\alpha$;
\item whenever $0<m<\theta$ and $s\in\Z^{([0,m+1))}$ is primitive
with $s_m\neq0$, the set
\[
 H_{m+1}\cap\chi_s^{-1}(t)
\]
is dense in $\chi_s^{-1}(t)$ for every $t\in\T$.
\end{enumerate}
\end{lemma}

\begin{proof}
The case $\theta=1$ is immediate.  Suppose $\theta>1$ and
$H=A\sqcup B$ is a separation into nonempty relatively clopen sets.
In $K_\theta$, let
\[
 U=K_\theta\setminus\cl B,
 \qquad V=K_\theta\setminus\cl A.
\]
These are disjoint nonempty open sets with $H\subseteq U\cup V$.
Define $U_\alpha,V_\alpha,W_\alpha$ by~\eqref{C:eq:overlaps}.
Hypothesis (i) says that the nonempty open sets $U_1,V_1$ cover
$\T$.  Since $\T$ is connected, $W_1\neq\varnothing$.

For $1\leq\alpha<\theta$, set
\[
 F_{\alpha+1}
 =K_{\alpha+1}\setminus(U_{\alpha+1}\cup V_{\alpha+1}).
\]
This is a compact set disjoint from $H_{\alpha+1}$.  By
Lemma~\ref{C:lem:obstacle},
$M_\alpha=\pi_\alpha^{\alpha+1}(F_{\alpha+1})$ is closed and
nowhere dense.  If $y\in W_\alpha\setminus M_\alpha$, the fiber
$(\pi_\alpha^{\alpha+1})^{-1}(y)$ is a circle, is covered by
$U_{\alpha+1},V_{\alpha+1}$, and meets both of them.  Connectedness
of that circle forces the two sets to overlap there.  Thus
\[
 W_\alpha\setminus M_\alpha
 \subseteq\pi_\alpha^{\alpha+1}(W_{\alpha+1}),
\]
which proves~\eqref{C:eq:successor-dense-overlap}.

Lemma~\ref{C:lem:limit-fusion} now gives a character-injective Cantor
set $P\subseteq K_\beta\setminus(U_\beta\cup V_\beta)$ for some
$\beta\leq\theta$.  Since $H_\beta\subseteq U_\beta\cup V_\beta$,
this contradicts hypothesis (ii).
\end{proof}

\begin{remark}\label{C:rem:inverse-limit}
The fusion argument takes place in the compact products
$K_\alpha$, not in an asserted inverse-limit representation of
$H$.  Its limit points are deliberately placed outside $H_\beta$.
What yields the contradiction is that they form an entire
character-injective Cantor set.  The argument therefore does not
assume that compatible points of the finite projections belong
to the actual subgroup.
\end{remark}

\begin{corollary}\label{C:cor:global-criterion}
Let $\kappa$ be a nonzero ordinal, and let $G\leq\T^\kappa$ be dense.
Suppose $p_{\{0\}}(G)=\T$, and suppose:
\begin{enumerate}
\item for every countable $J\subseteq\kappa$ containing $0$,
$p_J(G)$ meets every character-injective Cantor subset of $\T^J$;
\item for every nonzero primitive $s\in\Z^{(\kappa)}$ with
$m=\max\supp s>0$, and every $t\in\T$, the set
\[
 p_{[0,m]}(G)\cap\chi_s^{-1}(t)
\]
is dense in $\chi_s^{-1}(t)\subseteq\T^{[0,m]}$.
\end{enumerate}
Then $G$ is connected.
\end{corollary}

\begin{proof}
Give a countable $J$ containing $0$ its inherited order, whose
order type is a nonzero countable ordinal $\theta$.
Every prefix satisfies (i).  If a character involves the last
coordinate $m$ of a successor prefix, then $m$ is its largest
actual support coordinate.  Projecting (ii) to this prefix gives
hypothesis (iii) of Lemma~\ref{C:thm:criterion}; the entire character
fiber projects onto the fiber on the prefix because the support
is retained.  Thus $p_J(G)$ is connected.  Projections not
containing $0$ are continuous images of ones that do.  Apply
Lemma~\ref{C:lem:countable-support}(iii).
\end{proof}

The underlying group for these lemmas is the abstract circle group,
with the fixed decomposition~\eqref{C:eq:circle-decomposition}.
Fix, for every $0<i<\cc$, a countable divisible subgroup
$E_i\leq\T$ containing $D$ and a homomorphism $v_i:E_i\to\T$.
Only the map with index $i$ is prescribed on $E_i$.
The maps need not be continuous for the usual topology of $\T$.
For a family $(f_i)_{i<\cc}$ and $s\in\Z^{(\cc)}$, write
$f_s=\sum_i s_i f_i$; independence of integer vectors means
independence over $\Q$ in $\Q^{(\cc)}$.

At an intermediate step let
\[
 E=D\oplus U,\qquad U\leq W,\qquad |E|<\cc,
\]
be a common divisible domain.  In addition to the identity map,
we maintain homomorphisms
\begin{equation}\label{C:eq:partial-data}
 g_i:E+E_i\longrightarrow\T,
 \qquad g_i\res E_i=v_i \quad(0<i<\cc).
\end{equation}
Initially $E=D$ and $g_i=v_i$.  For a formal integer vector
$s\in\Z^{(\cc)}$, put
\[
 g_s=\sum_i s_i(g_i\res E),\qquad I_s=g_s(E),
\]
where $g_0$ is always the identity.  Each $I_s$ is divisible and
has cardinality at most $|E|$.  At this point a nonzero formal
vector may define the zero map on $E$; this causes no difficulty
because the initial common domain is countable.  After choosing
values on a new direction, we use $g_s$ also for the resulting sum
on the enlarged domain; $I_s$ continues to denote the old image
within that extension step.

Suppose $x\in W\setminus U$ is to be added to the common domain.
Define the set of \emph{forced coordinates} by
\begin{equation}\label{C:eq:forced-set}
 B(E,x)=\{i>0:x\in E+E_i\}.
\end{equation}
For $i\in B(E,x)$ the values on all of $\Q x$ are already fixed
by~\eqref{C:eq:partial-data}.  For $i\notin B(E,x)$ we have
\begin{equation}\label{C:eq:free-direct-sum}
 (E+E_i)\cap\Q x=\{0\},
\end{equation}
so any homomorphism on $\Q x$ can be adjoined by a direct-sum
extension.  Indeed, $E+E_i$ is divisible and contains $D$; by
\eqref{C:eq:divisible-domain}, its intersection with $W$ is a
$\Q$-subspace.  Formula~\eqref{C:eq:free-direct-sum} follows.

The set $B(E,x)$ can be large.  We will not try to control its
cardinality.  Instead we control, for each fixed coordinate, how
often it can be forced during the whole construction.

\begin{lemma}\label{C:lem:forced-count}
Consider an increasing transfinite sequence of common divisible
domains, each nontrivial successor step having the form
$E\subseteq E+\Q x$ with $x\in W\setminus E$.
For a fixed $i>0$, the condition $i\in B(E,x)$ holds at only
countably many steps.
\end{lemma}

\begin{proof}
At such a step write $x=e+c$ with $e\in E$ and $c\in E_i$.
Since $x\notin E$, we have $c\notin E$, whereas
$c=x-e\in E+\Q x$.  Thus
\[
 E_i\cap E\subsetneq E_i\cap(E+\Q x).
\]
The intersections with the fixed countable group $E_i$ increase
throughout the construction.  Every strict increase uses a new
element of $E_i$, so there can be only countably many of them.
\end{proof}

Consequently, for a fixed finitely supported vector $s$, the
condition
\begin{equation}\label{C:eq:forced-support}
 \supp s\cap B(E,x)\neq\varnothing
\end{equation}
occurs at only countably many steps.  We are allowed to enlarge
$\ker g_s$ at those steps.  This is the point at which simultaneous
extension of all the prescribed maps becomes compatible with
countable kernels.

We will repeatedly use the following elementary test.  If values
on $\Q x$ have been chosen so that
\begin{equation}\label{C:eq:no-new-kernel}
 g_s(qx)\notin I_s \qquad(q\in\Q\setminus\{0\}),
\end{equation}
then the extension of $g_s$ to $E+\Q x$ has no new kernel elements.
Indeed, $g_s(e+qx)=0$ would imply $g_s(qx)\in I_s$ unless $q=0$.
Even when new kernel elements are allowed, their quotient by the
old kernel embeds in
\begin{equation}\label{C:eq:kernel-quotient}
 (E+\Q x)/E\cong\Q.
\end{equation}
Thus a step adds at most countably many cosets of the old kernel. Indeed,
restriction of the quotient map $E+\Q x\to(E+\Q x)/E$ to the new
kernel has kernel equal to the old kernel. The first isomorphism theorem
gives the asserted embedding into $\Q$.

\begin{lemma}\label{C:lem:ordinary-extension}
Given the data~\eqref{C:eq:partial-data} and $x\in W\setminus U$,
there are extensions with common domain $E+\Q x$ which preserve
all the prescribed maps and add no kernel elements for any
nonzero $s$ satisfying
\begin{equation}\label{C:eq:unforced}
 \supp s\cap B(E,x)=\varnothing.
\end{equation}
\end{lemma}

\begin{proof}
Retain all the values at forced coordinates.  For a free coordinate
$i>0$, choose a real parameter $r_i$ and set
\[
 g_i(qx)=q r_i+\Z\qquad(q\in\Q).
\]
The zeroth coordinate remains $g_0(qx)=qx$.
For each $s$ satisfying~\eqref{C:eq:unforced} and having a nonzero
coefficient outside coordinate $0$, we require
\begin{equation}\label{C:eq:ordinary-forbidden}
 s_0qx+\sum_{i>0}s_i q r_i+\Z\notin I_s
 \quad(q\neq0).
\end{equation}
Choose the parameters at free coordinates in their ordinal order.
When $r_i$ is chosen, impose the requirements whose largest
nonzero nonzeroth coefficient has index $i$.  At this stage there
are at most $\max(\aleph_0,|i|)$ relevant vectors $s$.
For each such vector, nonzero $q\in\Q$, and prescribed value in
$I_s$, the equation forbidden in~\eqref{C:eq:ordinary-forbidden}
has at most countably many real solutions for $r_i$, because its
coefficient $s_iq$ is nonzero.  Therefore the total number of
forbidden parameters is at most
\begin{equation}\label{C:eq:parameter-bound}
 \max(\aleph_0,|i|,|E|)<\cc.
\end{equation}
A parameter remains available.

Use~\eqref{C:eq:free-direct-sum} to extend each free coordinate.
The forced coordinates already have the required domain.
Condition~\eqref{C:eq:ordinary-forbidden} is precisely the test
\eqref{C:eq:no-new-kernel} for the vectors under consideration.
A nonzero integer multiple of $g_0$ has its kernel in $D\subseteq E$
and does not acquire new kernel elements either.
\end{proof}

\begin{remark}
The conclusion of Lemma~\ref{C:lem:ordinary-extension} intentionally
omits vectors whose support meets $B(E,x)$.  Their prescribed
values need not leave a parameter available for
\eqref{C:eq:ordinary-forbidden}.  Lemma~\ref{C:lem:forced-count}, not
an additional avoidance condition, controls those possible kernel
increases.
\end{remark}

We next describe surjectivity, zero-image, and Cantor-set extensions.  A
surjectivity requirement imposes one linear equation on a new
direction.  A Cantor-set requirement prescribes countably many
coordinate values at a carefully chosen point.  In both cases the
remaining free coordinates are chosen as in
Lemma~\ref{C:lem:ordinary-extension}.

Fix $\tau=(n_i)\in\Z^{(\cc)}\setminus\Z e_0$, where $e_0$
is the unit vector at coordinate $0$.  Before treating this
requirement, assume that
\begin{equation}\label{C:eq:target-ready}
 E_i\subseteq E\quad(i\in\supp\tau,\ i>0).
\end{equation}
Thus every coordinate used by $\tau$, other than $0$, will be free
for every $x\in W\setminus U$.
Let $t\in\T\setminus I_\tau$, and choose a homomorphism
$b:\Q\to\T$ with $b(1)=t$.  For example, a real representative
of $t$ defines such a map.  Choose a pivot index $j>0$ with
$n_j\neq0$.  After choosing the other coordinate values on
$\Q x$, define
\begin{equation}\label{C:eq:pivot}
 g_j(qx)=b(q/n_j)-n_0(q/n_j)x
       -\sum_{i\neq0,j}n_i g_i((q/n_j)x).
\end{equation}
This is a homomorphism on $\Q x$, and it gives
\begin{equation}\label{C:eq:target-achieved}
 g_\tau(qx)=b(q),\qquad g_\tau(x)=t.
\end{equation}

\Needspace{5\baselineskip}
For an extension from $E$ to $E+\Q x$, write
\[
 I_\tau^{\mathrm{new}}:=g_\tau(E+\Q x),
\]
where $g_\tau$ denotes the extended map, while $I_\tau$
continues to denote its old image on $E$.

\begin{lemma}\label{C:lem:onto-extension}
Under~\eqref{C:eq:target-ready}, the point $x\in W\setminus U$ and
the extensions in~\eqref{C:eq:pivot} can be chosen so that:
\begin{enumerate}
\item  $t\in I_\tau^{\mathrm{new}}$;
\item if $s$ is independent of $\tau$ and satisfies
$\supp s\cap B(E,x)=\varnothing$, then $g_s$ acquires no new
kernel elements;
\item if $0\neq s\in\Z^{(\cc)}$ is proportional to $\tau$ and $t+I_\tau$ has infinite
order, then $g_s$ acquires no new kernel elements.
\end{enumerate}
\end{lemma}

\begin{proof}
For $s=(m_i)$ set
\[
 \delta_i=n_jm_i-m_jn_i.
\]
Multiplying the expression for $g_s(qx)$ by $n_j$ gives
\begin{equation}\label{C:eq:eliminated-pivot}
 n_jg_s(qx)=m_jb(q)+\delta_0qx
             +\sum_{i\neq0,j}\delta_i g_i(qx).
\end{equation}
First consider the independent vectors $s$ for which
$\delta_i=0$ at every $i\neq0,j$.  Such vectors belong to
\[
 \spanQ\{e_0,\tau\}\cap\Z^{(\cc)},
\]
a countable set, and $\delta_0\neq0$.  Choose $x\in W\setminus U$
so that for all of these vectors and all $q\in\Q\setminus\{0\}$,
\begin{equation}\label{C:eq:choose-direction}
 m_jb(q)+\delta_0qx\notin I_s.
\end{equation}
For fixed $s,q$, the map $x\mapsto\delta_0qx$ is one-to-one on
$W$, so at most $|I_s|$ choices of $x$ are forbidden.
In total fewer than $\cc$ choices are excluded.  Notice also that
the support of each of these $s$ is contained in
$\{0\}\cup\supp\tau$; by~\eqref{C:eq:target-ready} it contains no
forced coordinate.

Having chosen $x$, retain the forced coordinate values.  At every
other free coordinate $i\neq j$, put
$g_i(qx)=q r_i+\Z$ and choose $r_i$ successively.  For each
independent $s$ whose support avoids $B(E,x)$ and which has a
nonzero coefficient $\delta_i$ with $i\neq0,j$, require
\begin{equation}\label{C:eq:free-delta-avoidance}
 m_jb(q)+\delta_0qx
       +\sum_{i\neq0,j}\delta_i q r_i+\Z\notin I_s
 \quad(q\neq0).
\end{equation}
Assign each requirement to the largest index of a nonzero free
$\delta$-coefficient. At index $i$, any vector involved has support
in the free coordinates up to and including $i$, together with
the fixed finite set $\{0\}\cup\supp\tau$. The number of vectors is therefore at
most $\max(\aleph_0,|i|)$, and the same bound
\eqref{C:eq:parameter-bound} on forbidden real parameters applies.
There are no nonzero forced terms in~\eqref{C:eq:eliminated-pivot}
for these vectors: both $s$ and $\tau$ avoid $B(E,x)$.

Finally define the pivot by~\eqref{C:eq:pivot}.  It is a free
coordinate by~\eqref{C:eq:target-ready}, so the definition is
compatible with its prescribed map.  Part (1) follows from
\eqref{C:eq:target-achieved}. Conditions
\eqref{C:eq:choose-direction} and~\eqref{C:eq:free-delta-avoidance}
show that $n_jg_s(qx)\notin I_s$ in part (2).  Since $I_s$ is a
subgroup, this implies $g_s(qx)\notin I_s$, and
\eqref{C:eq:no-new-kernel} applies.

For (3), write $a s=c\tau$ with nonzero integers $a,c$.
The images of the old divisible domain are divisible. Multiplication by
a nonzero integer is therefore onto each old image, and $ag_s=cg_\tau$
gives
\begin{equation}\label{C:eq:proportional-images}
 I_s=aI_s=cI_\tau=I_\tau.
\end{equation}
If $g_s(qx)\in I_s$ and $q=p/l\neq0$, then
$a g_s(qx)=c b(q)\in I_\tau$, and multiplication by $l$ gives
$cp\,t\in I_\tau$.  This contradicts the infinite order of
$t+I_\tau$.
\end{proof}

Only one kind of surjectivity requirement can therefore enlarge
the kernels in its own rational line: the case
\begin{equation}\label{C:eq:torsion-growth}
 t\notin I_\tau,\qquad t+I_\tau\text{ has finite order}.
\end{equation}
Whenever this happens, choose a positive $k$ with $kt\in I_\tau$
and choose $a\in I_\tau$ with $ka=kt$.  Then
\begin{equation}\label{C:eq:new-torsion-image}
 t-a\in D\setminus I_\tau,
 \qquad t-a\in I_\tau^{\mathrm{new}}.
\end{equation}
Thus the intersection of the common image in that rational line
with the countable group $D$ strictly increases.  This will bound
the total number of such steps for any fixed line.

For a primitive vector $\tau\notin\Z e_0$, the next lemma uses
the extension formula of Lemma~\ref{C:lem:onto-extension} with $b=0$.
It produces a nonzero rational direction in the target kernel.

\begin{lemma}\label{C:lem:zero-extension}
Let $\tau\in\Z^{(\cc)}\setminus\Z e_0$ be primitive. Suppose $\alpha<\cc$, $\supp\tau\subseteq[0,\alpha]$, and
$E_i\subseteq E$ for every $0<i\leq\alpha$.
There is an extension on a direction $x\in W\setminus U$ such
that $g_\tau(\Q x)=\{0\}$ and, for every $s$ independent of $\tau$
whose support avoids $B(E,x)$,
\begin{equation}\label{C:eq:zero-independent}
 g_s(x)\notin D,
\end{equation}
while $g_s$ acquires no new kernel elements for all these $s$.
In particular,~\eqref{C:eq:zero-independent} holds for every such
$s$ supported in $[0,\alpha]$.
\end{lemma}

\begin{proof}
Repeat the proof of Lemma~\ref{C:lem:onto-extension} with $b=0$.
For the vectors with all free $\delta$-coefficients zero,
\eqref{C:eq:eliminated-pivot} gives
$n_jg_s(x)=\delta_0x$.  The right-hand side is a nonzero element
of $W$ and has infinite order.  Thus~\eqref{C:eq:zero-independent}
holds automatically for these vectors.

For the remaining vectors, in addition to
\eqref{C:eq:free-delta-avoidance} require the right-hand side of
\eqref{C:eq:eliminated-pivot} at $q=1$ not to belong to $D$.
This only adds countably many forbidden values for each vector,
so the parameter bound is unchanged.  If $g_s(x)$ were torsion,
then $n_jg_s(x)$ would be torsion, which has just been excluded.
The old avoidance conditions still prevent new kernel elements.
Finally, every nonzeroth coordinate at most $\alpha$ is free
because its prescribed domain is contained in $E$.
\end{proof}

For any full family $(f_i)_{i<\cc}$ extending the maps furnished by
Lemma~\ref{C:lem:zero-extension}, put $f(x)=(f_i(x))_{i<\cc}$.
The witness from that lemma satisfies
\begin{equation}\label{C:eq:prefix-density}
 \cl_{\T^{[0,\alpha]}}
 \{p_{[0,\alpha]}(f(qx)):q\in\Q\}
 =\ker\chi_\tau.
\end{equation}
Indeed, the integer characters vanishing on the displayed
subgroup are exactly $\Z\tau$: independent ones are excluded
by~\eqref{C:eq:zero-independent}, and primitive proportional ones
are the integer multiples of $\tau$.  Now apply
\eqref{C:eq:annihilator}.
Only this bounded prefix of coordinates is involved.  No density
in the kernel on the full final product is asserted.

\begin{lemma}\label{C:lem:cantor-extension}
Let $J\subseteq\cc$ be countable with $0\in J$, and let
$P\subseteq\T^J$ be a character-injective Cantor set.  Suppose
$E_i\subseteq E$ for every $i\in J\setminus\{0\}$.
There is an extension with common domain $E+\Q x$, for some
$x\in W\setminus U$, whose evaluation on this domain has a
$J$-projection meeting $P$.  Every nonzero vector $s\in\Z^{(\cc)}$ whose support avoids $B(E,x)$ acquires no new kernel elements.
\end{lemma}

\begin{proof}
For each nonzero $s\in\Z^{(J)}$, injectivity on $P$ gives
\[
 |P\cap\chi_s^{-1}(I_s+D)|\leq\max(\aleph_0,|E|).
\]
There are only countably many such vectors.  Since $|P|=\cc$,
choose $z\in P$ with
\begin{equation}\label{C:eq:cantor-choice}
 \chi_s(z)\notin I_s+D
 \quad(0\neq s\in\Z^{(J)}).
\end{equation}
For $s=e_0$ we have $I_{e_0}=E$ and $D\subseteq E$, so $z_0\notin E$.
Write $z_0=d+x$ with $d\in D$ and $x\in W\setminus U$.
All coordinates in $J\setminus\{0\}$ are free for this direction.
For each of them choose a homomorphism $h_i:\Q x\to\T$ with
\[
 h_i(x)=z_i-g_i(d),
\]
using Lemma~\ref{C:lem:divisible} on the cyclic subgroup $\Z x$.
Adjoin these maps to the respective $g_i$.  The evaluation at
$z_0=d+x$ now has $J$-projection $z$.

For any nonzero $s$ supported in $J$,
\[
 g_s(x)+I_s=\chi_s(z)+I_s.
\]
By~\eqref{C:eq:saturation} and~\eqref{C:eq:cantor-choice}, this coset
has infinite order.  If an element $e+qx$, with $e\in E$ and
$q=p/l\neq0$, belonged to the new kernel, then
\[
 p g_s(x)=l g_s(qx)\in I_s,
\]
a contradiction.  Hence these kernels do not increase.

At the forced coordinates outside $J$, keep the old values.
At the remaining coordinates outside $J$, put
$g_i(qx)=q r_i+\Z$.  Choose the parameters in ordinal order,
imposing~\eqref{C:eq:no-new-kernel} for the vectors whose support
avoids $B(E,x)$ but is not contained in $J$.  Assign a vector to
its largest nonzero coefficient outside $J$.  Its coefficients
on $J$ contribute fixed constants.  At index $i$ there are at most
$\max(\aleph_0,|J|,|i|)$ vectors to consider, so again fewer than
$\cc$ real parameters are forbidden.  This proves the remaining
kernel assertions and preserves all the prescribed maps.
\end{proof}

As in the ordinary extension, the conclusion does not claim that
\emph{every} kernel is unchanged at a Cantor-set step.  The possible
exceptions are exactly the vectors meeting $B(E,x)$; each fixed
vector has only countably many such exceptions over the entire
construction.

\begin{theorem}[Connected Baire reflexive example]\label{C:thm:main}
There exists in ZFC a dense subgroup $G\leq\T^{\cc}$ such that $G$ is
connected, Baire, and Pontryagin reflexive, every countable subgroup of
$G$ is $h$-embedded and closed, and every uncountable subgroup of $G$
is dense. Moreover, $G$ can be chosen algebraically isomorphic to $\T$,
every compact subset of $G$ is finite, every nonzero continuous character
of $G$ is surjective with countable kernel, and every nontrivial Hausdorff
quotient of $G$ is nonseparable.
\end{theorem}

\begin{proof}
The construction has three tasks. Prescribed extensions on countable
subgroups will give $h$-embeddedness. Surjective integer combinations
with countable kernels will ensure density of uncountable subgroups.
Cantor-set intersections and rational kernel witnesses will provide
connectedness; the same Cantor-set intersections will give the Baire
property. We carry out all these tasks in a single recursion and then
apply the compactness and duality results of Section~\ref{sec:prelim}.

\medskip\noindent\textbf{Step 1. Prescribed domains and target properties.}

Enumerate all pairs $(C_i,u_i)$, where $C_i\leq\T$ is countable and
$u_i:C_i\to\T$ is a homomorphism, by the indices $0<i<\cc$.
For each $i$, choose a countable divisible subgroup $E_i$ containing
$C_i\cup D$, and extend $u_i$ to $v_i:E_i\to\T$ by
Lemma~\ref{C:lem:divisible}. These are the data for the extension
lemmas. Each $v_i$ is prescribed only for its own coordinate; prescribing
all coordinates on the union of the $E_i$ would leave no room for the
geometric requirements.

We shall construct homomorphisms $f_i:\T\to\T$ with the following
properties:
\begin{enumerate}
\item\label{C:req:extension} $f_0=\id$ and $f_i\res E_i=v_i$ for $i>0$;
\item\label{C:req:kernels} every $f_s$ with $0\neq s\in\Z^{(\cc)}$ is surjective and
has a countable kernel;
\item\label{C:req:cantor} if $J\subseteq\cc$ is countable and contains $0$, then
\[
 \{(f_i(x))_{i\in J}:x\in\T\}
\]
meets every character-injective Cantor subset of $\T^J$;
\item\label{C:req:witness} if $\tau\in\Z^{(\cc)}$ is primitive and
$m=\max\supp\tau>0$, there is an $x_\tau\in W\setminus\{0\}$
such that $\Q x_\tau\subseteq\ker f_\tau$ and
\begin{equation}\label{C:eq:prefix-witness}
 f_s(x_\tau)\notin D
\end{equation}
for every $s\in\Z^{([0,m])}$ independent of $\tau$.
\end{enumerate}

\medskip\noindent\textbf{Step 2. Scheduling and the transfinite recursion.}

The scheduling ensures that the prescribed domains needed for each
geometric requirement are already in the common domain and that every
intermediate common domain has cardinality less than $\cc$.
Choose one primitive representative, with either sign fixed, from
each rational line other than $\Q e_0$.

Fix an enumeration $(w_\alpha)_{\alpha<\cc}$ of $W$.
In addition to importing every $E_i$ and every $w_\alpha$, use the
following requirements:
\begin{enumerate}
\item for every $\tau\in\Z^{(\cc)}\setminus\Z e_0$ and every
$t\in\T$, put $t$ in the image of $f_\tau$;
\item for each rational line other than $\Q e_0$, perform one
zero-image extension for its chosen primitive representative;
\item for every countable $J\subseteq\cc$ containing $0$ and every
character-injective Cantor set $P\subseteq\T^J$, make the
$J$-projection of the evaluation image meet $P$.
\end{enumerate}
There are at most $\cc$ requirements.  For the last type, there
are $\cc$ possible countable coordinate sets by
\eqref{C:eq:cardinals}, and each countable product has at most $\cc$
closed sets.  Assign distinct serial numbers $\xi(T)<\cc$ to the
requirements $T$.  Let $J_T$ be the finite support of the target
vector for the first two types and the countable coordinate set
for the third type.  Declare $T$ ready at stage
\begin{equation}\label{C:eq:ready-stage}
 r(T)=\max\{\xi(T),\sup J_T\}<\cc.
\end{equation}
The inequality follows from $\operatorname{cf}(\cc)>\aleph_0$.

Start with common domain $D$ and the maps $v_i$ on their individual
domains $E_i$.  At an outer stage $\alpha<\cc$, proceed as follows.
First, if $\alpha>0$, import all elements of $E_\alpha$ into the
common domain.  This requires at most countably many applications
of Lemma~\ref{C:lem:ordinary-extension}: write an element as $d+x$
using~\eqref{C:eq:circle-decomposition}, and add $\Q x$ if necessary.
Second, import $w_\alpha$ by the same lemma if it is not already
present.  Third, process all requirements $T$ with $r(T)=\alpha$,
in any fixed well-order.

For an image requirement, do nothing if its value already belongs
to the target image; otherwise use
Lemma~\ref{C:lem:onto-extension}.  For a zero-image requirement, use
Lemma~\ref{C:lem:zero-extension} once.  For a Cantor-set requirement,
use Lemma~\ref{C:lem:cantor-extension}.
At limit stages, including limits inside an outer stage, take
unions of the common domains and of the compatible coordinate
maps.  The domain of the $i$th map remains $E+E_i$, because
unions commute with adjoining the fixed subgroup $E_i$.

When a requirement is processed at stage $\alpha$, all the groups
$E_i$ with $0<i\leq\alpha$ have entered the common domain.
Thus~\eqref{C:eq:target-ready} and the hypotheses of both geometric
extension lemmas hold.  A zero-image requirement processed at
this stage gives the prefix witness for all coordinates at most
$\alpha$, and in particular for those at most the largest support
coordinate of its target.

There are at most $\max(\aleph_0,|\alpha|)$ requirements ready at
stage $\alpha$, because their distinct serial numbers are at most
$\alpha$.  The same bound applies to the total number of elementary
domain extensions performed through this stage: we have imported
at most $|\alpha+1|$ countable groups and processed requirements
whose serial numbers are at most $\alpha$.  Consequently, at
every intermediate step of this stage,
\begin{equation}\label{C:eq:domain-size-bound}
 |E|\leq\max(\aleph_0,|\alpha|)<\cc.
\end{equation}
Each elementary extension adds only one $\Q$-direction.  Choosing
all its coordinate parameters does not enlarge the common domain
further.  Estimate~\eqref{C:eq:domain-size-bound} verifies the
hypothesis $|E|<\cc$ of every extension lemma, even when $\cc$
is singular.  All requirements are eventually processed.

\medskip\noindent\textbf{Step 3. Countability of every nonzero character kernel.}

The common domains exhaust $\T$ because every element of $W$ is
imported.  Hence the unions of the coordinate maps define
homomorphisms $f_i:\T\to\T$ with $f_i\res E_i=v_i$.
We verify the kernel estimate before drawing any topological
conclusions.

Fix $0\neq s\in\Z^{(\cc)}$.  A nonzero multiple of $e_0$ has
finite kernel contained in $D$, so suppose $s\notin\Z e_0$.
Initially its kernel on the common domain $D$ is countable.
An elementary extension can add kernel elements only in the
following circumstances:
\begin{enumerate}
\item the support of $s$ meets $B(E,x)$;
\item it is a surjectivity extension whose target is proportional
to $s$, and the new target value has nonzero finite order modulo
the old image;
\item it is the single zero-image extension belonging to the
rational line of $s$.
\end{enumerate}
The list is exhaustive by Lemmas~\ref{C:lem:ordinary-extension},
\ref{C:lem:onto-extension}, \ref{C:lem:zero-extension},
and~\ref{C:lem:cantor-extension}.

Steps of the first type are countable in number by
Lemma~\ref{C:lem:forced-count} and the finiteness of $\supp s$.
For the second type, all proportional vectors have the same old
image at every step by~\eqref{C:eq:proportional-images}.  Each such
step strictly increases its intersection with $D$ by
\eqref{C:eq:new-torsion-image}.  These intersections are increasing,
so a fixed rational line has only countably many steps of this
type.  There is only one step of the third type.

It follows that the kernel of this fixed vector can increase at
only countably many elementary steps. Let $\mathcal S_s$ be the
well-ordered set of these steps. For each $\eta\in\mathcal S_s$,
write $K_\eta^-$ and $K_\eta^+$ for the kernels immediately before
and after that step. By~\eqref{C:eq:kernel-quotient},
$K_\eta^+/K_\eta^-$ embeds in $\Q$, so choose a countable set
$R_\eta\subseteq K_\eta^+$ of representatives for its cosets.
Let $K_0$ be the initial kernel on $D$. Then
\[
 \ker f_s=\gen{K_0\cup\bigcup_{\eta\in\mathcal S_s}R_\eta}.
\]
Indeed, transfinite induction over the elementary extensions proves
that every intermediate kernel is contained in the group on the
right: at a growth step, subtract a representative to obtain an
element of the old kernel; at other successor steps the kernel is
unchanged. At a limit step the kernel is the union of the earlier
kernels, since each element already belongs to an earlier domain
and its character value never changes. The reverse inclusion holds
because all chosen elements retain value zero. The displayed
generating set is countable, so the final kernel is countable.

This count is stronger than an estimate at a single stage.
Merely adding countably many kernel elements at each of $\cc$
stages would not give the desired conclusion.  Here the stages
at which a \emph{fixed} kernel can increase are themselves
countable.

\medskip\noindent\textbf{Step 4. Verification of the coordinate requirements.}
The preceding construction preserves all the maps $v_i$, and
$f_0=\id$ throughout, giving property~\ref{C:req:extension}.  Every image requirement is
satisfied either when it is processed or earlier.  Thus each
$f_s$ with $s\notin\Z e_0$ is onto $\T$; nonzero multiples of
$f_0$ are onto as well.  The kernel count proves property~\ref{C:req:kernels}.
All the Cantor-set requirements were processed, giving property~\ref{C:req:cantor}.

Finally, the zero-image requirement for a primitive representative
$\tau$ is processed at some stage
$\alpha\geq\max\supp\tau$.  By
Lemma~\ref{C:lem:zero-extension}, its direction $x_\tau$ is nonzero,
$\Q x_\tau\subseteq\ker f_\tau$, and
\eqref{C:eq:prefix-witness} holds on the entire required prefix.
Later steps never change these values.  The other choice of sign
for a primitive representative uses the same witness.  This proves property~\ref{C:req:witness}.

Define the evaluation map and its image by
\begin{equation}\label{C:eq:evaluation-group}
 f:\T\longrightarrow\T^{\cc},\qquad
 f(x)=(f_i(x))_{i<\cc},\qquad G=f(\T).
\end{equation}
Equip $G$ with the subspace topology.

\medskip\noindent\textbf{Step 5. Density, continuous characters, and subgroups.}

The map $f$ is injective because its zeroth coordinate is the
identity.  Thus the underlying abstract group of $G$ is $\T$ and
$|G|=\cc$.  Every nonzero integer character of $\T^{\cc}$ restricts
to an onto map on $G$, by property~\ref{C:req:kernels}.
In particular no such character vanishes on $G$, so
\eqref{C:eq:annihilator} shows that $G$ is dense in $\T^{\cc}$.
It is therefore precompact with this compact completion.
Every continuous character of $G$ extends to the completion and
is the restriction of an integer character.  Consequently all
its nonzero continuous characters are onto and have countable
kernels.

Let $H\leq G$ be countable and let $u:H\to\T$ be an abstract
homomorphism. Under the algebraic identification $f:\T\to G$, the
pair $(f^{-1}(H),u\circ f|_{f^{-1}(H)})$ is one of the prescribed
requests $(C_i,u_i)$. Property~\ref{C:req:extension} shows that the
continuous coordinate character $p_{\{i\}}|_G$ extends $u$.
Hence every countable subgroup of $G$ is $h$-embedded.

Lemma~\ref{lem:closed} shows that every countable subgroup is closed.

If an uncountable subgroup $L\leq G$ were not dense, its closure
in $\T^{\cc}$ would be a proper closed subgroup.  By
\eqref{C:eq:annihilator}, a nonzero integer character would vanish
on $L$.  This is impossible because its kernel on $G$ is
countable.  Thus every uncountable subgroup of $G$ is dense.

Lemma~\ref{lem:quotient-obstruction} proves that all nontrivial Hausdorff quotients of $G$ are nonseparable.

\medskip\noindent\textbf{Step 6. Connectedness.}

Property~\ref{C:req:cantor} gives the first hypothesis
of Corollary~\ref{C:cor:global-criterion}.  For its second hypothesis,
fix a primitive $\tau$ with $m=\max\supp\tau>0$.
The witness $x_\tau$ in property~\ref{C:req:witness} and the annihilator calculation
in~\eqref{C:eq:prefix-density} give
\begin{equation}\label{C:eq:final-prefix-kernel}
 \cl_{\T^{[0,m]}}p_{[0,m]}(f(\Q x_\tau))
 =\ker\chi_\tau.
\end{equation}
For $t\in\T$, choose $e\in\T$ with $f_\tau(e)=t$.
The projection of the coset $f(e+\Q x_\tau)$ is then dense in
$\chi_\tau^{-1}(t)\subseteq\T^{[0,m]}$ and is contained in the
corresponding level of $p_{[0,m]}(G)$.
Also $p_{\{0\}}(G)=\T$.
Corollary~\ref{C:cor:global-criterion} proves that $G$ is connected.

\medskip\noindent\textbf{Step 7. The Baire property and Pontryagin reflexivity.}
Let $J\subseteq\cc$ be nonempty and countable and let
$R\subseteq\T^J$ be a dense $G_\delta$-set. Put $J'=J\cup\{0\}$.
The inverse image of $R$ under $\T^{J'}\to\T^J$ is dense
$G_\delta$. Lemma~\ref{C:lem:cantor-gdelta} supplies a
character-injective Cantor set in this inverse image.
Property~\ref{C:req:cantor} shows that $p_{J'}(G)$ meets
it, so $p_J(G)$ meets $R$. Lemma~\ref{B:lem:baire} therefore shows
that $G$ is Baire. Since its countable subgroups are $h$-embedded,
Lemma~\ref{B:lem:countable} shows that every compact subset of $G$ is
finite. Corollary~\ref{B:cor:reflexive} now gives Pontryagin reflexivity.
This completes the proof.
\end{proof}

Let $G$ and the coordinate maps $f_i$ be those
constructed in the proof of Theorem~\ref{C:thm:main}.
We record several consequences which clarify the relation
between connectedness, small kernels, and the absence of separable
quotients.

\begin{proposition}\label{C:prop:nonopen}
Every nonzero continuous character of $G$ is a nonopen
surjection.  If $\overline\chi:\T^{\cc}\to\T$ denotes its extension,
then
\begin{equation}\label{C:eq:strict-kernel-closure}
 \cl_{\T^{\cc}}(\ker\chi)\subsetneq\ker\overline\chi.
\end{equation}
\end{proposition}

\begin{proof}
The subgroup $\chi^{-1}(D)$ is countable, because $D$ and
$\ker\chi$ are countable.  It is proper, because $\chi$ is onto,
and it is closed by the main theorem.  If $\chi$ were open, the
density of $D$ in $\T$ would make $\chi^{-1}(D)$ dense in $G$,
a contradiction.

For completeness, equality in~\eqref{C:eq:strict-kernel-closure}
would force $\chi$ to be open.  If $V\subseteq\T^{\cc}$ is open
and $t\in\overline\chi(V)$, choose $g\in G$ with $\chi(g)=t$.
The coset $g+\ker\chi$ would be dense in the compact fiber
$\overline\chi^{-1}(t)$, and so would meet $V$.  Hence
\[
 \chi(V\cap G)=\overline\chi(V).
\]
The right-hand side is open, since a continuous surjective
homomorphism between compact groups is open.  This contradicts
the first assertion.
\end{proof}

\begin{remark}
Formula~\eqref{C:eq:final-prefix-kernel} does not contradict
\eqref{C:eq:strict-kernel-closure}.  A countable kernel can have a
dense projection in the character kernel on an earlier coordinate
prefix, while later coordinates prevent density in the full
character kernel.  The construction only requires the former.
\end{remark}

\begin{proposition}\label{C:prop:finite-kernels}
The nonzero continuous characters of $G$ with finite kernel are
exactly the nonzero integer multiples of its zeroth coordinate.
If two continuous characters are independent over $\Q$, their
kernels are not commensurable.
\end{proposition}

\begin{proof}
For a nonzero integer $n$, the map $nf_0$ has kernel $\T[n]=\{t\in\T:nt=0\}$,
which is finite.  Any other integer vector $s$ is an integer
multiple of a primitive representative $\tau$ with nonzero
nonzeroth support.  Its kernel contains the infinite subgroup
$\Q x_\tau$ and is therefore countably infinite.

Two subgroups are commensurable if their intersection has finite
index in each.  Let $s,r$ be independent vectors.  Choose, say,
$s$ so that its largest support coordinate is at least the largest
support coordinate of $r$, and let $\tau$ be its primitive
representative.  Then $x_\tau\in\ker f_s$, while
$f_r(x_\tau)$ has infinite order by~\eqref{C:eq:prefix-witness}.
The first isomorphism theorem now rules out commensurability:
\[
 [\ker f_s:\ker f_s\cap\ker f_r]=\aleph_0.\qedhere
\]
\end{proof}

\begin{proposition}\label{C:prop:separable-subsets}
Every separable subspace of $G$ is countable.  In particular,
every nonempty separable connected subspace is a singleton, and $G$ has
no nonconstant paths.
\end{proposition}

\begin{proof}
If $Y\subseteq G$ is separable, choose a countable set $A\subseteq Y$
dense in $Y$.  The subgroup $C=\langle A\rangle$ is countable and
closed in $G$.  Therefore
\[
 Y\subseteq\cl_G A\subseteq C,
\]
so $Y$ is countable.  A connected Tychonoff space with two distinct
points has a continuous real-valued function taking distinct
values at those points; its image contains a nondegenerate interval.
Such a space is uncountable.  Every subspace of a Hausdorff
topological group is Tychonoff, so a nonempty separable connected subspace
of $G$ must be a singleton.  The image of a path is separable and
connected, proving the final assertion.
\end{proof}

\begin{remark}
Theorem~\ref{C:thm:main} answers all three parts of
Problem~\ref{prob:regularity} affirmatively and simultaneously:
the same precompact group is connected, Baire, and Pontryagin
reflexive, with the required properties of countable and uncountable
subgroups. The category and duality results used in the proof are
established in Section~\ref{sec:prelim} independently of either
construction.
\end{remark}

\section{A free Baire Pontryagin-reflexive example}\label{sec:free}
We retain the conventions and definitions of Section~\ref{sec:prelim}.
In particular, $A(X)$ denotes the abstract free abelian group on $X$.
We use the character formula~\eqref{C:eq:integer-character}, the
annihilator formula~\eqref{C:eq:annihilator}, extension to the compact
completion, and Lemma~\ref{C:lem:divisible} without repeating their proofs.

Let $P\leq\T$. A labelled family $(t_j)_{j\in J}$ is
\emph{independent modulo $P$} if
\begin{equation}\label{B:eq:independence}
 \sum_{j\in J}n_jt_j\in P
 \quad\Longrightarrow\quad n_j=0\text{ for all }j\in J
\end{equation}
for every finitely supported integer family $(n_j)_{j\in J}$.
Such a family freely generates a subgroup which intersects $P$ only
at zero. Independence without qualification means independence modulo
$\{0\}$.

\begin{lemma}[Independent selection in a dense $G_\delta$-set]\label{B:lem:selection}
Let $J$ be nonempty and countable, let $R\subseteq\T^J$ be a dense
$G_\delta$-set, and let $P\leq\T$ have cardinality less than $\cc$.
There exists $t\in R$ whose coordinates $(t(j))_{j\in J}$ are
independent modulo $P$.
\end{lemma}

\begin{proof}
Put $Z=\T^J$ and write $R=\bigcap_{n\geq1}U_n$, where each $U_n$
is dense and open. We first construct a family
$(z_\eta)_{\eta\in2^{\N}}\subseteq R$ such that the entire labelled
family
\begin{equation}\label{B:eq:cantor-independent}
 \bigl(z_\eta(j)\bigr)_{(\eta,j)\in2^{\N}\times J}
\end{equation}
is independent.

For $m\geq1$ and a nonzero finitely supported integer array
$a=(a_{i,j})_{1\leq i\leq m,\,j\in J}$, let
\[
 F_a=\left\{(z_1,\ldots,z_m)\in Z^m:
       \sum_{i=1}^m\sum_{j\in J}a_{i,j}z_i(j)=0\right\}.
\]
This is the kernel of a nonzero continuous character of $Z^m$.
It is closed and nowhere dense: a subgroup with nonempty interior is
open, and the connected group $Z^m$ has no proper open subgroup.
There are only countably many relations $F_a$; list them as
$(F_r)_{r\geq1}$, retaining the arity $m_r$ of each relation.

Fix a compatible complete metric on the perfect compact space $Z$.
We construct nonempty open sets $V_s$, indexed by finite binary
strings, with $V_{\varnothing}=Z$. At level $n\geq1$ their closures
are pairwise disjoint, have diameter at most $2^{-n}$, and satisfy
\[
 \overline{V_s}\subseteq
 V_{s|_{n-1}}\cap\bigcap_{r=1}^nU_r
 \qquad(|s|=n).
\]
We also require that, for $r\leq n$ and distinct strings
$s_1,\ldots,s_{m_r}$ of length $n$,
\begin{equation}\label{B:eq:fusion}
 \overline{V_{s_1}}\times\cdots\times\overline{V_{s_{m_r}}}
 \quad\text{is disjoint from }F_r.
\end{equation}

Here is the induction step. Choose all centres for the next level
simultaneously in the finite product of their prescribed parent open
sets, intersected with the first $n$ dense open sets. The forbidden
relations in this product are pullbacks of finitely many $F_r$ under
open coordinate projections, so they are nowhere dense. The conditions
that two centres coincide also form nowhere dense closed sets, since
$Z$ has no isolated points. We can therefore choose distinct centres
avoiding all forbidden relations. Shrinking their neighbourhoods gives
the required closure inclusions, diameter bounds and
\eqref{B:eq:fusion}. There are only finitely many conditions at each
level.

For $\eta\in2^{\N}$, let $z_\eta$ be the unique point of
$\bigcap_n\overline{V_{\eta|_n}}$. It belongs to $R$.
Any proposed nonzero integer relation among finitely many entries
of~\eqref{B:eq:cantor-independent} can be grouped according to the
distinct branches that occur. It is then one of the relations $F_r$.
At a sufficiently deep level the branches have separated and $r\leq n$,
contradicting~\eqref{B:eq:fusion}. This proves the claimed independence.

Let $E$ be the set of all entries in~\eqref{B:eq:cantor-independent}.
Every element of $\gen{E}$ has a unique finite integer expression
in these labelled entries. For each $p\in P\cap\gen{E}$, discard
all branches appearing in its expression. Fewer than $\cc$ branches
are discarded, because each expression is finite and $|P|<\cc$.
Choose a remaining branch $\eta$. A nonzero integer combination of
the coordinates of $z_\eta$ cannot belong to $P$: its unique expression
would force the branch $\eta$ to have been discarded. Hence
$t=z_\eta$ has the desired property.
\end{proof}

The first part of this proof is a specific instance of Mycielski's
independent-set method~\cite{Mycielski}. The second part permits a
subgroup $P$ of any cardinality less than $\cc$. In particular, the
proof does not use an assertion that fewer than $\cc$ meagre sets
cannot cover a Polish space.

For later individual choices, recall the saturation notation from
Section~\ref{sec:prelim}. For an arbitrary subgroup $P\leq\T$,
observe that
\begin{equation}\label{B:eq:saturation}
 |\Sat(P)|\leq\max(\aleph_0,|P|).
\end{equation}
Indeed, the equation $nt=p$ has exactly $n$ solutions in $\T$. Thus,
when $|P|<\cc$, an element of $\T\setminus\Sat(P)$ can always be
chosen; its singleton family is independent modulo $P$.

\begin{lemma}[Countable kernels of finite addition maps]\label{B:lem:sum}
Let $(H_i)_{i\in I}$ be subgroups of $\T$ with $H_i\subseteq B_i+L_i$,
where each $B_i$ is countable and the subgroups $L_i\leq\T$ form an
algebraic direct sum. For every finite nonempty $F\subseteq I$, the
addition map
\[
 s_F:\prod_{i\in F}H_i\longrightarrow\T,
 \qquad s_F((t_i))=\sum_{i\in F}t_i
\]
has countable kernel.
\end{lemma}
\begin{proof}
Write each coordinate of a kernel tuple as $t_i=b_i+l_i$. There are
countably many choices for $(b_i)$. For each such choice,
$\sum_i l_i=-\sum_i b_i$ has at most one solution tuple, because
addition on $\prod_{i\in F}L_i$ is injective. This counts all kernel
tuples even if the individual decompositions are not unique.
\end{proof}

\begin{lemma}[Countable kernels of integer combinations]\label{B:lem:all-kernels}
Let $A$ be a torsion-free abelian group, and let
$(g_i:A\to\T)_{i\in I}$ be homomorphisms with countable kernels.
Suppose that $g_i(A)\subseteq B_i+L_i$, where each $B_i$ is countable
and the subgroups $L_i\leq\T$ form an algebraic direct sum.
Then every nonzero finitely supported integer family $(n_i)_{i\in I}$
satisfies
\[
 \left|\ker\left(\sum_{i\in I}n_i g_i\right)\right|\leq\aleph_0.
\]
\end{lemma}

\begin{proof}
Put $Q=\{i\in I:n_i\neq0\}$ and $H_i=g_i(A)$.
For $a$ in the displayed kernel, the tuple
$(g_i(n_i a))_{i\in Q}$ belongs to the kernel of the addition map
$s_Q$ from Lemma~\ref{B:lem:sum}. That kernel is countable.
Fix $i_0\in Q$. There are therefore only countably many possible
values of $g_{i_0}(n_{i_0}a)$. Every fibre of $g_{i_0}$ is countable,
so there are only countably many possibilities for $n_{i_0}a$.
Multiplication by $n_{i_0}\neq0$ is injective on the torsion-free
group $A$, and the conclusion follows.
\end{proof}

\begin{theorem}\label{B:thm:main}
There exists in ZFC a dense subgroup $G\leq\T^{\cc}$ whose underlying
group is free abelian of rank $\cc$, which is zero-dimensional, Baire,
and Pontryagin reflexive, whose countable subgroups are $h$-embedded and closed, and
whose uncountable subgroups are dense. Every compact subset of $G$ is
finite, and every nontrivial Hausdorff quotient of $G$ is nonseparable.
\end{theorem}

\begin{proof}
We construct coordinate homomorphisms on a free abelian group.
Two families of requirements are imposed simultaneously: extension
of every homomorphism on a countable subgroup, and meeting every dense
$G_\delta$-set in every countable coordinate projection. Independent
free entries make all nonzero integer combinations of coordinates
have countable kernels. These three ingredients yield, respectively,
$h$-embeddedness, the Baire property, and density of uncountable
subgroups. The coordinate images are proper subgroups of the circle,
which gives zero-dimensionality. The common preliminaries then give
reflexivity and the quotient obstruction.

\medskip\noindent\textbf{Step 1. Prescribed homomorphisms and category requirements.}

Let $X=\{x_\xi:\xi<\cc\}$ and let $F=A(X)$. We identify cardinals
with their initial ordinals. Since $|F|=\cc$ and
$\cc^{\aleph_0}=\cc$, the pairs consisting of a countable subgroup
$C\leq F$ and a homomorphism $u:C\to\T$ can be enumerated as
\begin{equation}\label{B:eq:hom-list}
 \{(C_\alpha,u_\alpha):\alpha<\cc\}.
\end{equation}
For each $\alpha$, choose a countable $Y_\alpha\subseteq X$ with
$C_\alpha\subseteq A(Y_\alpha)$. Fix an extension
\[
 h_\alpha:A(Y_\alpha)\longrightarrow\T,
 \qquad h_\alpha|_{C_\alpha}=u_\alpha,
\]
and put $B_\alpha=h_\alpha(A(Y_\alpha))$. Each $B_\alpha$ is
countable. All these data are fixed before the recursion begins.

There are also exactly $\cc$ pairs $(J,R)$ for which
$\varnothing\neq J\subseteq\cc$ is countable and $R$ is a dense
$G_\delta$ subset of $\T^J$. Indeed, there are $\cc$ countable
subsets of $\cc$, and a second countable space has at most $\cc$
$G_\delta$-sets. Enumerate these pairs as
\begin{equation}\label{B:eq:category-list}
 \{(J_\xi,R_\xi):\xi<\cc\}.
\end{equation}

We shall define a matrix $(a_{\alpha,x})_{\alpha<\cc,\,x\in X}$
of elements of $\T$. The entries in row $\alpha$ determine a
homomorphism $f_\alpha:F\to\T$. Entries with $x\in Y_\alpha$
are prescribed by $h_\alpha$; the remaining entries will be chosen
independently.

\medskip\noindent\textbf{Step 2. Partial matrices and extension operations.}

At any point of the recursion, let $I\subseteq\cc$ be the set of
introduced rows and let $S\subseteq X$ be the set of introduced
basis elements. The domain of the partial matrix is
\begin{equation}\label{B:eq:domain}
 \{(\alpha,x):\alpha\in I,\ x\in S\cup Y_\alpha\}.
\end{equation}
We maintain the following conditions:
\begin{enumerate}
\item[(i)] $a_{\alpha,x}=h_\alpha(x)$ for $\alpha\in I$ and
$x\in Y_\alpha$;
\item[(ii)] the labelled family of free entries
\[
 \mathcal E=\bigl(a_{\alpha,x}\bigr)_
 {\alpha\in I,\,x\in S\setminus Y_\alpha}
\]
is independent;
\item[(iii)] for each $\alpha\in I$, the subgroup generated by its
free entries intersects $B_\alpha$ only at zero.
\end{enumerate}
In choosing a new free entry, we normally let $P$ be the subgroup of
$\T$ generated by all existing free entries together with
$\bigcup_{\alpha\in I}B_\alpha$. Provided $|P|<\cc$, choosing
outside $\Sat(P)$ preserves (ii) and (iii). The group $P$ is updated
after each choice.

To introduce a new row $\alpha$, first prescribe all its entries on
$Y_\alpha$ using $h_\alpha$. This row had no entries before, so no
previous value changes. Include $B_\alpha$ among the generators of
$P$. For each $x\in S\setminus Y_\alpha$, in any well-order,
choose $a_{\alpha,x}$ outside the current $\Sat(P)$. This produces
a partial matrix with introduced row set $I\cup\{\alpha\}$.

To introduce a new basis element $x$, retain the entries in rows for
which $x\in Y_\alpha$, since these are already prescribed. In each
remaining introduced row choose a new free entry by the same procedure.
The resulting introduced basis set is $S\cup\{x\}$.

There is no requirement that all existing free entries be independent
modulo the union of all the groups $B_\alpha$. A newly introduced
$B_\alpha$ may contain old free entries from other rows. This causes
no conflict: it does not alter a relation among the old free entries,
and row $\alpha$ acquires its first free entry only after $B_\alpha$
has been included in the choice of $P$. Thereafter $B_\alpha$ is fixed.
This observation proves that the extension operations preserve exactly
the three stated conditions.

\medskip\noindent\textbf{Step 3. The recursion and its cardinal bounds.}

Start with $I=S=\varnothing$. At stage $\xi<\cc$, perform the
following operations in order.

First introduce all rows in $\{\xi\}\cup J_\xi$ which have not
yet appeared. Next introduce $x_\xi$ if necessary. Choose a basis
element
\begin{equation}\label{B:eq:witness}
 w_\xi\in X\setminus
 \left(S\cup\bigcup_{\alpha\in I}Y_\alpha\right).
\end{equation}
Thus no value in an introduced row is prescribed in advance for
$w_\xi$. Let $P$ be generated by the current free entries and the
groups $B_\alpha$, $\alpha\in I$. Apply
Lemma~\ref{B:lem:selection} to choose $t_\xi\in R_\xi$ with
coordinates independent modulo $P$, and prescribe simultaneously
\begin{equation}\label{B:eq:prescribed-category}
 a_{\alpha,w_\xi}=t_\xi(\alpha)
 \qquad(\alpha\in J_\xi).
\end{equation}
For the rows in $I\setminus J_\xi$, fill the remaining entries
of $w_\xi$ one at a time outside the current $\Sat(P)$, including
the entries just chosen in~\eqref{B:eq:prescribed-category} among the
generators of $P$. Finally add $w_\xi$ to $S$.

Independence modulo $P$ shows that the simultaneous choice preserves
(ii). It also preserves (iii), since $P$ contains the old free
entries and all introduced $B_\alpha$. The subsequent individual
choices preserve the same conditions. At a limit stage take the union
of the preceding partial matrices. Every possible failure of (ii) or
(iii) would involve a finite relation and hence would already occur
at an earlier stage.

For completeness, all choices above have the required cardinal bounds.
At stage $\xi$, put $\kappa_\xi=\max(\aleph_0,|\xi|)<\cc$.
Only countably many rows and at most two basis elements are introduced
at each stage. Therefore, throughout stage $\xi$, both $I$ and $S$
have cardinality at most $\kappa_\xi$. The domain in
\eqref{B:eq:domain} has cardinality at most
$\kappa_\xi^2=\kappa_\xi$, and so do the free entries and
$\bigcup_{\alpha\in I}B_\alpha$. Consequently $|P|<\cc$ at every
individual or simultaneous choice. The excluded set in
\eqref{B:eq:witness} also has cardinality at most $\kappa_\xi$, so
$w_\xi$ exists. These estimates use a single bound for each stage
and do not assume that $\cc$ is regular. Rows in $J_\xi$ are
introduced when needed; no condition such as $J_\xi\subseteq\xi$
is required.

Every row $\alpha$ and every basis element $x_\alpha$ has appeared
by the end of stage $\alpha$. Hence the recursion defines the full
matrix. Extend its rows to homomorphisms and set
\[
 f_\alpha:F\longrightarrow\T,\quad f_\alpha(x)=a_{\alpha,x},
 \qquad
 f=(f_\alpha)_{\alpha<\cc}:F\longrightarrow K=\T^{\cc}.
\]
For each $\alpha<\cc$, put
\begin{equation}\label{B:eq:E}
 E_\alpha=\{f_\alpha(x):x\in X\setminus Y_\alpha\}.
\end{equation}
Condition (ii) implies that the sets $E_\alpha$ are pairwise disjoint
and that their union $E$ is independent. Condition (iii) gives
\begin{equation}\label{B:eq:E-B}
 \gen{E_\alpha}\cap B_\alpha=\{0\}.
\end{equation}
The other properties of the resulting homomorphisms are
\begin{align}
 f_\alpha|_{C_\alpha}&=u_\alpha,
       \label{B:eq:extension}\\
 \ker f_\alpha&\subseteq A(Y_\alpha),
       \label{B:eq:small-kernel}\\
 f_\alpha(F)&=B_\alpha+\gen{E_\alpha},
       \label{B:eq:image}\\
 p_{J_\xi}(f(w_\xi))&=t_\xi\in R_\xi.
       \label{B:eq:category-met}
\end{align}
To verify~\eqref{B:eq:small-kernel}, write $a\in F$ as $a=y+z$,
where $y\in A(Y_\alpha)$ and
$z\in A(X\setminus Y_\alpha)$. If $f_\alpha(a)=0$, then
$f_\alpha(z)\in\gen{E_\alpha}\cap B_\alpha=\{0\}$.
Independence of the free entries in row $\alpha$ gives $z=0$.
In particular, every $\ker f_\alpha$ is countable. Finally,
\eqref{B:eq:category-met} persists because the recursion never changes
an assigned entry.

\medskip\noindent\textbf{Step 4. Injectivity and countable subgroups.}

Let $G=f(F)$ with the topology inherited from $K=\T^{\cc}$.

If $0\neq a\in F$, choose a homomorphism
$u:\gen{a}\to\T$ with $u(a)\neq0$. The pair $(\gen{a},u)$
occurs in~\eqref{B:eq:hom-list}. For the corresponding $\alpha$,
\eqref{B:eq:extension} gives $f_\alpha(a)\neq0$. Thus $f$ is
injective and $G$ is free abelian of rank $\cc$.

Let $D\leq G$ be countable and let $v:D\to\T$ be an abstract
homomorphism. The pair
\[
 \bigl(f^{-1}(D),\ v\circ f|_{f^{-1}(D)}\bigr)
\]
also appears in~\eqref{B:eq:hom-list}. The corresponding coordinate
character $p_{\{\alpha\}}|_G$, with $\T^{\{\alpha\}}$ identified
with $\T$, extends $v$. Hence every countable subgroup of $G$ is
$h$-embedded. Lemmas~\ref{lem:closed}
and~\ref{B:lem:countable} show that all these subgroups are closed and
that all compact subsets of $G$ are finite.

\medskip\noindent\textbf{Step 5. Character kernels and density.}

Put $H_\alpha=f_\alpha(F)$ and $L_\alpha=\gen{E_\alpha}$.
Equation~\eqref{B:eq:image} gives
$H_\alpha=B_\alpha+L_\alpha$, with $B_\alpha$ countable, and
condition (ii) makes the subgroups $L_\alpha$ an algebraic direct sum.
Also, every $\ker f_\alpha$ is countable by
\eqref{B:eq:small-kernel}. Applying Lemma~\ref{B:lem:all-kernels}
with $A=F$ and $g_\alpha=f_\alpha$, and using the character
formula~\eqref{C:eq:integer-character}, we obtain
\begin{equation}\label{B:eq:all-character-kernels}
 |\ker(\psi\circ f)|\leq\aleph_0
 \quad\text{for every nonzero continuous character }\psi:K\to\T.
\end{equation}

If $G$ were not dense in $K$, a nonzero character of
$K/\overline G^{\,K}$ would pull back to a nonzero character of
$K$ vanishing on $G$. This contradicts~\eqref{B:eq:all-character-kernels},
because $|G|=\cc$. Hence $G$ is dense and precompact.

The same argument applies to an uncountable subgroup $D\leq G$.
If $D$ were not dense in $G$, its closure in $K$ would be a proper
closed subgroup of $K$. A nonzero character annihilating that
closure would have uncountable kernel on $G$, again a contradiction.
Therefore every uncountable subgroup of $G$ is dense. In fact, since
characters of $G$ extend to $K$, the argument also proves that every
nonzero continuous character of $G$ has countable kernel.

\medskip\noindent\textbf{Step 6. The Baire property and Pontryagin reflexivity.}

Equation~\eqref{B:eq:category-met} realizes every pair in
\eqref{B:eq:category-list}. Thus the dense subgroup $G$ satisfies
\eqref{B:eq:category}, and Lemma~\ref{B:lem:baire} shows that $G$ is
Baire. We have already proved that every compact subset of $G$ is
finite. Corollary~\ref{B:cor:reflexive} now implies that $G$ is
Pontryagin reflexive.

\medskip\noindent\textbf{Step 7. Zero-dimensionality and the quotient obstruction.}

For each $\alpha<\cc$, equations~\eqref{B:eq:image}
and~\eqref{B:eq:E-B} give the algebraic isomorphism
\[
 H_\alpha/B_\alpha\cong\gen{E_\alpha}.
\]
The group on the right is a nonzero free abelian group: $Y_\alpha$
is countable, $|X|=\cc$, and the entries indexed by
$X\setminus Y_\alpha$ are independent. If $H_\alpha=\T$, the
quotient $H_\alpha/B_\alpha$ would be divisible, whereas a nonzero
free abelian group is not divisible. Thus $H_\alpha$ is a proper
subgroup of $\T$.

Every proper subgroup $H$ of $\T$, with its subspace topology, is
zero-dimensional. Indeed, $H$ has empty interior: otherwise it
would be an open subgroup of the connected group $\T$, and hence
would equal $\T$. Consequently $\T\setminus H$ is dense.
Given $h\in H$ and a neighbourhood $O$ of $h$ in $\T$, choose an
open arc $I$ containing $h$, contained in $O$, and with both
endpoints outside $H$. The set $I\cap H$ is open in $H$; it is
also closed in $H$, because
$\cl_{\T}I\setminus I$ consists of those two endpoints.
These intersections form a base of open-and-closed sets for $H$.

It follows that every $H_\alpha$ is zero-dimensional. The topology
of $G$ inherited from $\prod_{\alpha<\cc}H_\alpha$ agrees with its
topology inherited from $\T^{\cc}$. Products and subspaces of
spaces with open-and-closed bases again have such bases, so $G$
is zero-dimensional.

Finally, the subgroup properties already proved allow us to apply
Lemma~\ref{lem:quotient-obstruction}. Every nontrivial Hausdorff
quotient of $G$ is therefore nonseparable. This completes the proof.
\end{proof}

In the following discussion, $G$, $H_\alpha$, $B_\alpha$, and
$E_\alpha$ retain the meanings assigned in the proof of
Theorem~\ref{B:thm:main}.
A separable quotient in the statements below is a Hausdorff quotient
endowed with its quotient topology.

Theorem~\ref{B:thm:main} gives an infinite Pontryagin-reflexive
abelian group with no nontrivial separable Hausdorff quotient.
Consequently, the answers to parts~(i), (ii), and~(iv) of
Problem~\ref{prob:reflexive} are negative: an infinite separable
quotient, with or without metrizability, would in particular be a
nontrivial separable quotient. The answer to part~(iii) is also
negative if the metrizable quotient is required to be nontrivial.
The connected group of Theorem~\ref{C:thm:main} gives the same answers.

\begin{remark}
The quotient topology is essential. The group constructed here has
nonzero continuous characters into $\T$, so it has nontrivial
separable metrizable continuous homomorphic images with their subspace
topologies. These maps are not quotient maps onto those images.
There is therefore no conflict between the quotient obstruction and
the existence of the coordinate characters used in the construction.
\end{remark}

\begin{remark}
Theorem~\ref{B:thm:main} retains zero-dimensionality as well as the
countable and uncountable subgroup properties of the example
in~\cite[Theorem~3.5]{LMT}, and simultaneously adds the Baire
property and Pontryagin reflexivity. It therefore answers
parts~(b) and~(c) of Problem~\ref{prob:regularity} even with
zero-dimensionality retained. In particular, this group is totally
disconnected. The connected example in Theorem~\ref{C:thm:main}
answers all three parts of that problem simultaneously.
\end{remark}

\section{Closed subgroups and quotients of finite powers}\label{sec:powers}

The coordinate information recalled below is the input for the proof.
After two auxiliary lemmas, Theorem~\ref{P:thm:structure} establishes
both the closed-subgroup normal form and the quotient conclusion in
one argument. We then apply the same reasoning to the free example of
Section~\ref{sec:free}.

We retain the conventions of Section~\ref{sec:prelim} and the notation $A(X)$ for the abstract free abelian group on $X$.

We recall the parts of the construction in~\cite[Theorem~3.5]{LMT} needed
below. Let $X$ be a set of cardinality $\cc$ and let $A(X)$ be the free
abelian group on $X$. The construction provides a monomorphism
\[
  f:A(X)\longrightarrow\T^{\cc},\qquad G=f(A(X)),
\]
whose image is dense. For $\alpha<\cc$, let $\pi_\alpha$ be the
$\alpha$th coordinate projection, and put
\[
  f_\alpha=\pi_\alpha\circ f,\qquad
  q_\alpha=\pi_\alpha|_G,\qquad H_\alpha=q_\alpha(G).
\]
There is a countable set $Y_\alpha\subseteq X$ such that
$\ker f_\alpha\subseteq\gen{Y_\alpha}$. Since $f$ is injective,
\begin{equation}\label{P:eq:coordinate-kernel}
  \ker q_\alpha\text{ is countable for every }\alpha<\cc.
\end{equation}

The construction also uses a set $E\subseteq\T$ whose elements have
infinite order and are independent over $\Z$, together with a partition
$E=\bigcup_{\alpha<\cc}E_\alpha$. We use independence in the sense of~\eqref{B:eq:independence},
with $P=\{0\}$. Set
$B_\alpha=\gen{f_\alpha(Y_\alpha)}$. Condition~(iv) of the construction
gives
\begin{equation}\label{P:eq:coordinate-decomposition}
  H_\alpha\subseteq B_\alpha+\gen{E_\alpha},
  \qquad |B_\alpha|\leq\aleph_0.
\end{equation}
In particular, the subgroups $\gen{E_\alpha}$ form an algebraic direct sum
inside $\T$. Finally, $G$ is torsion-free, since it is isomorphic to
$A(X)$, and every countable subgroup of $G$ is $h$-embedded in $G$.

Equation~\eqref{P:eq:coordinate-decomposition} and the independence of
the subgroups $\gen{E_\alpha}$ verify the hypotheses of
Lemma~\ref{B:lem:sum}. Hence every finite-coordinate addition map
$s_F$ has countable kernel. This is also the counting argument
in~\cite[proof of Theorem~3.11, Claim~3]{LMT}.

For $m=(m_1,\ldots,m_k)\in\Z^k$ and $x=(x_1,\ldots,x_k)\in G^k$,
write $m\cdot x=\sum_{i=1}^k m_i x_i$. If $M\leq G^k$, write
$m\cdot M=\{m\cdot x:x\in M\}$.

Every continuous character $\chi:G^k\to\T$ extends to the compact
completion $(\T^{\cc})^k$. The characters of a product of circles are
finite integer combinations of coordinate projections. Therefore
\begin{equation}\label{P:eq:character}
  \chi(x)=\sum_{\alpha\in F_\chi}
       q_\alpha(m_{\chi,\alpha}\cdot x),
  \qquad m_{\chi,\alpha}\in\Z^k,
\end{equation}
where $F_\chi\subseteq\cc$ is finite. We may omit zero rows; the zero
character then corresponds to $F_\chi=\varnothing$. These standard facts
about completions and characters of precompact groups can be found
in~\cite{AT}.

\begin{lemma}[Countable images of character rows]\label{P:lem:row}
Let $\chi$ be a continuous character of $G^k$, written as
in~\eqref{P:eq:character}. For every $\alpha\in F_\chi$, the subgroup
$m_{\chi,\alpha}\cdot\ker\chi$ of $G$ is countable.
\end{lemma}

\begin{proof}
For $x\in\ker\chi$, the tuple
\[
  \bigl(q_\alpha(m_{\chi,\alpha}\cdot x)\bigr)_{\alpha\in F_\chi}
\]
belongs to $\ker s_{F_\chi}$. By Lemma~\ref{B:lem:sum}, its $\alpha$th
coordinate takes only countably many values. Every fibre of $q_\alpha$
is countable by~\eqref{P:eq:coordinate-kernel}. Hence the set of possible
values of $m_{\chi,\alpha}\cdot x$ is countable.
\end{proof}

We need two permanence facts. The finite-product assertion
is~\cite[Lemma~3.9]{LMT}; the quotient assertion is the countable-kernel
argument in~\cite[proof of Lemma~3.8]{LMT}. Their short proofs are
included for completeness.

\begin{lemma}[Permanence of countable $h$-embeddedness]\label{P:lem:permanence}
The following assertions hold for Hausdorff abelian topological groups.
\begin{enumerate}[label=\textup{(\arabic*)}]
\item If every countable subgroup of each of $A_1,\ldots,A_r$ is
$h$-embedded, where $r\geq1$ is finite, then the same is true of
$A_1\times\cdots\times A_r$.
\item If every countable subgroup of $A$ is $h$-embedded and $C$ is a
countable closed subgroup of $A$, then every countable subgroup of $A/C$
is $h$-embedded in $A/C$.
\end{enumerate}
\end{lemma}

\begin{proof}
(1) Let $D\leq\prod_{i=1}^r A_i$ be countable, and let
$u:D\to\T$ be an abstract homomorphism. Write $D_i$ for the $i$th
coordinate projection of $D$. By Lemma~\ref{C:lem:divisible}, $u$ extends to a homomorphism
$v:\prod_{i=1}^r D_i\to\T$. Since the product is finite,
there are homomorphisms $v_i:D_i\to\T$ such that
\[
  v(d_1,\ldots,d_r)=\sum_{i=1}^r v_i(d_i).
\]
Extend each $v_i$ to a continuous character of $A_i$ and add the resulting
coordinate characters. This gives the required extension of $u$.

(2) Let $q:A\to A/C$ be the quotient map. If $D\leq A/C$ is
countable, then $q^{-1}(D)$ is countable. Given $u:D\to\T$, extend
$u\circ q|_{q^{-1}(D)}$ to a continuous character $\chi$ of $A$.
The character $\chi$ vanishes on $C$, so it descends to a continuous
character of $A/C$ extending $u$.
\end{proof}

\begin{theorem}\label{P:thm:structure}
Let $G$ be the group obtained in the proof of~\cite[Theorem~3.5]{LMT},
let $k\geq1$, and let $N$ be a closed subgroup of $G^k$. There exist
$0\leq r\leq k$, a matrix $U\in\GL_k(\Z)$, and a countable closed
subgroup $C\leq G^r$ such that $\Phi_U(N)=C\times G^{k-r}$, where
$\Phi_U$ is the coordinate automorphism induced by $U$. Consequently,
$G^k/N\cong G^r/C$ topologically. Moreover, every countable subgroup
of every Hausdorff quotient of $G^k$ is $h$-embedded and closed.
\end{theorem}

\begin{proof}
The proof extracts the integer rows whose images on $N$ are countable.
They form a saturated sublattice, so an integral change of coordinates
turns them into coordinate projections. Character separation then
identifies $N$ as the inverse image of a countable closed subgroup.
The permanence lemma transfers $h$-embeddedness to the resulting
quotient.

\medskip\noindent\textbf{Step 1. Characters separating the closed subgroup.}

For the given closed subgroup $N\leq G^k$, put
\[
  N^\perp=\{\chi\in\widehat{G^k}:\chi|_N=0\}.
\]
The Hausdorff quotient $G^k/N$ is precompact. Its continuous characters
separate points, as follows by embedding it into its compact completion.
Consequently,
\begin{equation}\label{P:eq:annihilator}
  N=\bigcap_{\chi\in N^\perp}\ker\chi.
\end{equation}

\medskip\noindent\textbf{Step 2. The saturated lattice and its coordinate map.}
Define
\begin{equation}\label{P:eq:lattice}
  L_N=\{m\in\Z^k:|m\cdot N|\leq\aleph_0\}.
\end{equation}
This is a subgroup of $\Z^k$: for $m,n\in L_N$,
\[
  (m-n)\cdot N\subseteq(m\cdot N)-(n\cdot N),
\]
and the set on the right is countable. Moreover, $L_N$ is saturated.
Indeed, if $d\neq0$ is an integer and $dm\in L_N$, then
$d(m\cdot N)$ is countable. Multiplication by $d$ is injective on the
torsion-free group $G$, so $m\cdot N$ is countable and $m\in L_N$.

It follows that $\Z^k/L_N$ is a finitely generated torsion-free abelian
group, hence free. Thus $L_N$ is a direct summand of $\Z^k$. Choose a
basis $b_1,\ldots,b_r$ of $L_N$ and extend it to a basis
$b_1,\ldots,b_k$ of $\Z^k$. Let $U\in\GL_k(\Z)$ have these vectors
as its rows, and define
\[
  p:G^k\longrightarrow G^r,
  \qquad p(x)=(b_1\cdot x,\ldots,b_r\cdot x).
\]
The map $p$ is the composition of $\Phi_U$ with the projection onto the
first $r$ coordinates. In particular, $p$ is an open continuous
surjection. Put $K=\ker p$ and $C=p(N)$. Since $b_i\in L_N$ for
$i\leq r$,
\begin{equation}\label{P:eq:countable-image}
  C\subseteq\prod_{i=1}^r(b_i\cdot N)
\end{equation}
is countable. When $r=0$, this assertion means $C=\{0\}$.

\medskip\noindent\textbf{Step 3. Recovering the entire subgroup from the coordinate map.}

We claim that $K\subseteq N$. Let $\chi\in N^\perp$, and write $\chi$
as in~\eqref{P:eq:character}. Since $N\subseteq\ker\chi$,
Lemma~\ref{P:lem:row} gives
\[
  m_{\chi,\alpha}\cdot N
     \subseteq m_{\chi,\alpha}\cdot\ker\chi,
  \qquad |m_{\chi,\alpha}\cdot N|\leq\aleph_0.
\]
Thus every row $m_{\chi,\alpha}$ belongs to $L_N$ and is an integer
linear combination of $b_1,\ldots,b_r$. If $x\in K$, all these rows
vanish on $x$, and hence $\chi(x)=0$. Equation~\eqref{P:eq:annihilator}
now yields $x\in N$, proving the claim.

\medskip\noindent\textbf{Step 4. The normal form and the topological quotient.}

As $K\subseteq N$, we have
\begin{equation}\label{P:eq:inverse-image}
  p^{-1}(C)=p^{-1}(p(N))=N.
\end{equation}
The map $p$ is a quotient map and $N$ is closed, so $C$ is closed in
$G^r$. In the coordinates given by $\Phi_U$,
equation~\eqref{P:eq:inverse-image} becomes
\[
  \Phi_U(N)=C\times G^{k-r}.
\]
Finally, the composition of $p$ with the quotient map $G^r\to G^r/C$
is an open continuous surjection with kernel $N$. It therefore induces
the asserted topological isomorphism $G^k/N\cong G^r/C$.

\medskip\noindent\textbf{Step 5. Countable subgroups of the quotient.}

The case $r=0$ gives the trivial quotient and is immediate.
Suppose $r\geq1$. Every countable subgroup of $G$ is $h$-embedded
by the recalled construction. Lemma~\ref{P:lem:permanence}(1)
therefore gives the same property for $G^r$. Since $C$ is countable
and closed, part~(2) of that lemma shows that every countable subgroup
of $G^r/C$ is $h$-embedded. Lemma~\ref{lem:closed} gives its
closedness. The topological isomorphism from Step~4 transfers both
properties to $G^k/N$. Every Hausdorff quotient of $G^k$ has this
form, so the proof is complete.
\end{proof}

The last assertion of Theorem~\ref{P:thm:structure} answers both
alternatives in Problem~\ref{prob:powers} affirmatively. In particular,
countable subgroups of the quotients in~\cite[Problem~3.12]{LMT}
are both $h$-embedded and closed.

\begin{corollary}\label{P:cor:square}
If $N$ is an uncountable proper closed subgroup of $G^2$, then there are
$U\in\GL_2(\Z)$ and a countable closed subgroup $C\leq G$ such that
$\Phi_U(N)=C\times G$. In particular, $G^2/N\cong G/C$ topologically.
\end{corollary}

\begin{proof}
In Theorem~\ref{P:thm:structure}, the case $r=0$ gives $N=G^2$, while
$r=2$ makes $N$ countable. Thus $r=1$.
\end{proof}

\begin{remark}\label{P:rem:hypotheses}
Steps~1--4 of the proof of Theorem~\ref{P:thm:structure} establish
the structure assertion without using $h$-embeddedness. They apply
to any torsion-free dense subgroup $A\leq\T^I$ for which each
coordinate projection has countable kernel and the addition map
$\prod_{\alpha\in F}\pi_\alpha(A)\to\T$ has countable kernel for
every finite nonempty $F\subseteq I$. Under these assumptions, for
every closed $N\leq A^k$ there are $U\in\GL_k(\Z)$ and a countable
closed $C\leq A^r$ with $\Phi_U(N)=C\times A^{k-r}$ and
$A^k/N\cong A^r/C$.

The conclusion about $h$-embedded countable subgroups of these quotients
requires the additional assumption that every countable subgroup of
$A$ is $h$-embedded, together with Lemma~\ref{P:lem:permanence}.
For example, if $\alpha\in\T$ has infinite order, the dense cyclic
subgroup $A=\gen{\alpha}\leq\T$ satisfies the one-coordinate kernel
conditions and is torsion-free. It is not $h$-embedded in itself:
choose $t\in\T\setminus\{n\alpha:n\in\Z\}$. The homomorphism
$k\alpha\mapsto kt$ is not continuous, since a continuous character
of $A$ extends to $\T$ and hence must be the restriction of an integer
multiple map. Thus the structural hypotheses alone do not imply the
$h$-embeddedness conclusion.

For the original group $G$, the independence condition
in~\eqref{P:eq:coordinate-decomposition} supplies the addition-kernel
hypothesis. In particular, the argument uses features of the construction,
rather than only the conclusions stated in~\cite[Theorem~3.5]{LMT}.
\end{remark}

\begin{corollary}[Finite powers of the free Baire reflexive example]
\label{P:cor:free-powers}
Let $G$ be the group constructed in Theorem~\ref{B:thm:main}.
For every integer $k\geq1$ and every closed subgroup $N\leq G^k$,
there exist $0\leq r\leq k$, $U\in\GL_k(\Z)$, and a countable
closed subgroup $C\leq G^r$ such that
\[
 \Phi_U(N)=C\times G^{k-r},\qquad G^k/N\cong G^r/C
\]
topologically. Every countable subgroup of every Hausdorff quotient
of $G^k$ is $h$-embedded and closed.
\end{corollary}
\begin{proof}
Use the coordinate maps and the sets $B_\alpha,E_\alpha$ from the
proof of Theorem~\ref{B:thm:main}. The group $G$ is torsion-free and
dense in $\T^{\cc}$ by that theorem. Each coordinate kernel is countable
by~\eqref{B:eq:small-kernel}. Its coordinate images satisfy
\[
 H_\alpha=B_\alpha+\gen{E_\alpha},\qquad |B_\alpha|\leq\aleph_0,
\]
by~\eqref{B:eq:image}, and the subgroups $\gen{E_\alpha}$ form an
algebraic direct sum, by the independence of all free entries.
Lemma~\ref{B:lem:sum} therefore gives countable kernels for all
finite-coordinate addition maps. The structure conclusion now follows
from Remark~\ref{P:rem:hypotheses}. Finally, all countable subgroups of
$G$ are $h$-embedded. Applying both parts of
Lemma~\ref{P:lem:permanence} to $G^r/C$, and then
Lemma~\ref{lem:closed}, proves the quotient assertion. The case
$r=0$ is immediate.
\end{proof}

Return to the original group $G$ of Theorem~\ref{P:thm:structure}.
The closedness assertion can also be expressed directly in terms of
subgroups of $G^k$. The following corollary and example also apply to
the free group of Theorem~\ref{B:thm:main}, by
Corollary~\ref{P:cor:free-powers}.

\begin{corollary}\label{P:cor:sum-closed}
If $N\leq G^k$ is closed and $D\leq G^k$ is countable, then $N+D$
is closed in $G^k$.
\end{corollary}

\begin{proof}
Use the open homomorphism $p:G^k\to G^r$ from the proof of
Theorem~\ref{P:thm:structure}, and put $C=p(N)$. Since $\ker p\subseteq N$,
\[
  N+D=p^{-1}\bigl(C+p(D)\bigr).
\]
The subgroup $C+p(D)$ is countable and therefore closed in $G^r$ by
Lemmas~\ref{P:lem:permanence}(1) and~\ref{lem:closed}. Its inverse image is closed.
\end{proof}

We finish by noting that the quotient assertion of Theorem~\ref{P:thm:structure} cannot be extended to
infinite powers. Both properties already fail for countable subgroups of
$G^\omega$ itself.

\begin{example}\label{P:ex:infinite}
Choose $0\neq g\in G$, and put $A=\gen{g}$. The subgroup $A$ is
countable and closed in $G$. Let
\[
  D=\bigoplus_{n\in\omega}A\leq G^\omega,
  \qquad z=(g,g,\ldots),\qquad B=D+\Z z.
\]
Both $D$ and $B$ are countable. Since finite-support tuples are dense in
$A^\omega$ and $A^\omega$ is closed in $G^\omega$,
\[
  \overline D^{\,G^\omega}
    =\overline B^{\,G^\omega}=A^\omega.
\]
The group $A^\omega$ is uncountable, so neither $D$ nor $B$ is closed.
Moreover, $D\cap\Z z=\{0\}$, because $g$ has infinite order. Fix
$0\neq t\in\T$ and define $u:B\to\T$ by
$u(d+nz)=nt$. If $u$ extended to a continuous character of $G^\omega$,
that character would vanish on $D$ and hence on $A^\omega$, contradicting
$u(z)=t$. Thus $B$ is not $h$-embedded in $G^\omega$.
\end{example}

\end{document}